%% file: preprint.tex
\documentclass[a4paper,11pt]{article}
\usepackage[utf8]{inputenc}
\usepackage[english]{babel}

\usepackage{natbib}
\usepackage{hyperref}
\usepackage{amsmath}
\usepackage{amssymb}
\usepackage{amsthm}
\usepackage{accents}
\usepackage{subcaption}
\usepackage{graphicx}
\usepackage{tikz}
\usepackage{xcolor}
\usepackage[linesnumbered]{algorithm2e}
\usepackage{todonotes}
\usepackage{mathtools}
\usepackage{multirow}
\usepackage{booktabs}
\usepackage{bbm}

\newtheorem{theorem}{Theorem}
\newtheorem{lemma}[theorem]{Lemma}
\newtheorem{corollary}[theorem]{Corollary}
\theoremstyle{definition}
\newtheorem{example}{Example}
\newtheorem{definition}{Definition}
\newtheorem{remark}{Remark}

\newcommand{\problemabbrv}{RSRP-PdM}
\newcommand{\graphname}{SEEG}
\newcommand{\completegraphname}{CEEG}

\newcommand{\reals}{\mathbb{R}}
\newcommand{\integer}{\mathbb{Z}}
\newcommand{\pos}[1]{{#1}_{\geq 0}}
\newcommand{\strictpos}[1]{#1_{> 0}}
\newcommand{\binary}{\{0,1\}}
\DeclareMathOperator*{\argmin}{arg\,min}

\DeclareMathOperator*{\erf}{erf}
\DeclareMathOperator*{\disjointUnion}{\Dot{\bigcup}}
\newcommand{\vars}{s}
\newcommand{\varsi}[1]{\vars_{#1}}
\newcommand{\vart}{t}
\newcommand{\varti}[1]{\vart_{#1}}
\newcommand{\cost}{c}
\newcommand{\varx}{x}

\newcommand{\vard}{d}

\newcommand{\variance}{\sigma^2}
\newcommand{\domain}{B}
\newcommand{\directions}{R}

\newcommand{\trips}{\mathcal{T}}
\newcommand{\trip}{t}
\newcommand{\vehicles}{\mathcal{V}}
\newcommand{\vehicle}{v}
\newcommand{\paramspace}{\Theta}
\newcommand{\param}{\theta}
\newcommand{\degradation}[1]{\Delta_{#1}}
\newcommand{\locations}{\mathcal{L}}
\newcommand{\location}{l}
\newcommand{\maintenances}{\locations_M}
\newcommand{\maintenance}{m}

\newcommand{\timept}{k}
\newcommand{\discr}{\mathcal{D}}

\newcommand{\nodes}{V}
\newcommand{\node}{v}
\newcommand{\floor}[1]{\lfloor {#1} \rfloor}
\newcommand{\floorto}[2]{\floor{#1}_{#2}}

\newcommand{\deptime}[1]{\timept_{#1}^d}
\newcommand{\arrtime}[1]{\timept_{#1}^a}
\newcommand{\deploc}[1]{\location_{#1}^d}
\newcommand{\arrloc}[1]{\location_{#1}^a}
\newcommand{\nvehicle}[1]{n_{#1}^v}
\newcommand{\arcs}{A}
\newcommand{\arc}{a}
\newcommand{\graph}{G}
\newcommand{\health}[2]{H_{#1,#2}}
\newcommand{\timepts}{\mathcal{K}}
\newcommand{\sbv}[1]{e_{#1}}
\newcommand{\ones}[1]{\mathbbm{1}_{#1}}
\newcommand{\artificial}{\Phi}
\newcommand{\prob}{\mathbb{P}}
\newcommand{\failureprob}{\prob_{f}}
\newcommand{\outgoing}{\delta^{+}}
\newcommand{\incoming}{\delta^{-}}
\newcommand{\normal}[2]{\mathcal{N}(#1, #2)}

\newcommand{\best}[1]{\textbf{#1}}

\newcommand{\setft}[2]{\{#1,\dots,#2\}}
\newcommand{\eucldist}[1]{\|{#1}\|_2}
\newcommand{\sgnleq}[1]{\mathop{\preceq}_{#1}}

\newcommand{\cube}[1]{\mathcal{E}_{#1}}
\newcommand{\deriv}[1]{\frac{\partial}{\partial{#1}}}
\newcommand{\conv}[1]{\mathit{conv}(#1)}

\newcommand{\emittingcone}[2]{\mathcal{C}_{#1}(#2)}
\newcommand{\paths}{\mathcal{P}}
\newcommand{\otherparam}{\varphi}
\newcommand{\gammadist}[2]{\Gamma(#1, #2)}

\newcommand{\identity}{\text{id}_{\reals}}
\newcommand{\error}{\varepsilon}
\newcommand{\orig}{\varphi}
\newcommand{\rounded}{\tau}
\newcommand{\abs}[1]{\left|{#1}\right|}

\newcommand{\powerset}[1]{2^{#1}}
\newcommand{\pois}{\mathit{Pois}}
\newcommand{\weibull}{\mathit{Weibull}}
\newcommand{\pdf}{\Pi}
\newcommand{\diffbar}[1]{C^1(#1)}
\definecolor{zibblue}{HTML}{03869F}

\DeclareRobustCommand{\orcidlink}[1]{%
  \href{https://orcid.org/#1}{\begingroup\normalfont%
    \raisebox{-\fontchardp\font`q}{%
      \includegraphics[height=\fontcharht\font`/+\fontchardp\font`q]{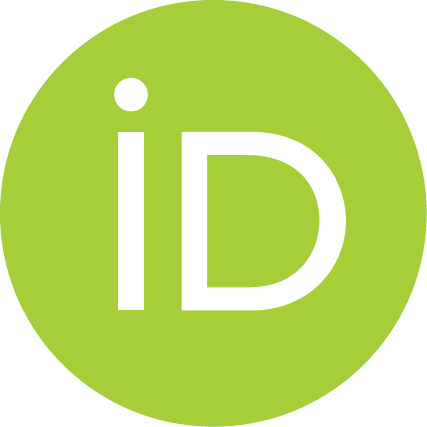}%
    }\endgroup%
    ~#1%
  }%
}
\newcommand{\orcid}[1]{\footnote{\orcidlink{#1}}}

\pgfdeclarelayer{bg} 
\pgfsetlayers{bg,main}

\title{An Iterative Refinement Approach for the \\Rolling Stock Rotation Problem with \\Predictive Maintenance}
\author{%
\begin{tabular}{c}
Felix Prause\orcid{0000-0001-9401-3707} \ and Ralf Bornd\"{o}rfer\orcid{0000-0001-7223-9174}\\
Zuse Institute Berlin\\
Takustr.~7, 14195 Berlin, Germany\\
\texttt{prause@zib.de}, \texttt{borndoerfer@zib.de}
\end{tabular}%
}
\date{\today}

\begin{document}

\maketitle

\begin{abstract}
The rolling stock rotation problem with predictive maintenance (\problemabbrv) involves the assignment of trips to a fleet of vehicles with integrated maintenance scheduling based on the predicted failure probability of the vehicles.
These probabilities are determined by the health states of the vehicles, which are considered to be random variables distributed by a parameterized family of probability distribution functions. During the operation of the trips, the corresponding parameters get updated.
In this article, we present a dual solution approach for \problemabbrv\ and generalize a linear programming based lower bound for this problem to families of probability distribution functions with more than one parameter.
For this purpose, we define a rounding function that allows for a consistent underestimation of the parameters and model the problem by a state-expanded event-graph in which the possible states are restricted to a discrete set. This induces a flow problem that is solved by an integer linear program.
We show that the iterative refinement of the underlying discretization leads to solutions that converge from below to an optimal solution of the original instance. Thus, the linear relaxation of the considered integer linear program results in a lower bound for \problemabbrv.
Finally, we report on the results of computational experiments conducted on a library of test instances.
\end{abstract}


\section{Introduction}
\label{sec:introduction}
The current advances in machine learning and the associated ability to efficiently analyze large volumes of data, combined with the inexpensive availability of sensors, enable the effective monitoring of the health conditions of mechanical components.
This allows predictive maintenance to be established as a regime for maintenance scheduling in the rail sector.
Predictive maintenance is one of the four maintenance systems that exist in the literature, alongside corrective, preventive, and condition-based maintenance.
Corrective and preventive maintenance are the two most obvious techniques, where a component is maintained whenever a fault occurs or when a certain time or distance mileage is reached.
Condition-based and predictive maintenance, on the other hand, rely on estimating the health state of the respective component. This is usually done by evaluating sensor measurements using statistical methods or machine learning techniques.
In condition-based maintenance, components are serviced when the estimated health state exceeds a predefined threshold, whereas in predictive maintenance, the remaining useful life (RUL) or future health states of the component are predicted and maintenance is scheduled just before the estimated occurrence of a fault.
For an overview of these different maintenance systems, we refer to \cite{wang2007selection}.
It is evident that predictive maintenance is the most economical and environmentally friendly of these regimes, see for example \cite{arena2021predictive}, since the components are maintained before failure and the remaining useful life is completely utilized so that no unnecessary replacement actions are carried out.

Unfortunately, health states can only be observed indirectly and must therefore be estimated from sensor measurements. Errors arising from these measurements and the subsequent estimation lead to the fact that the determined health states must be regarded as uncertain. This uncertainty is then further increased by projecting the states into the future under uncertain operating conditions, so that both the health states themselves and their predictions must be considered as random variables.
Thus, if we want to incorporate predictive maintenance into the planning of rolling stock rotations, we need to develop approaches that are capable of optimizing the rotations based on these uncertain circumstances, see \cite{bougacha2020contribution}.

\subsection{Related Work}
\label{subsec:relatedWork}
The rolling stock rotation problem (RSRP) with integrated maintenance scheduling is a topic that has already been investigated in the literature.
An overview can be found in \cite{reuther2018optimization}. The authors claim that there is no standard way to address maintenance constraints for RSRP and refer to \cite{reuther2017mathematical} for examples.
The considered maintenance regime is typically distant- or time-dependent preventive maintenance, and the problem is usually modeled by event-based graphs. These graphs are also called space-time graphs and give rise to solution approaches based on mixed-integer programs (MIP) that determine flows within these graphs. These MIP formulations are then solved using either branch-and-price or column generation, e.g., \cite{borndorfer2014coarse,giacco2014rolling,andres2015maintenance,borndorfer2016integrated,lusby2017branch,gao2022weekly}. Furthermore, also heuristics are used to solve RSRP, see for example \cite{cacchiani2010solving,thorlacius2015integrated}.
For a comparison of the various considered maintenance regimes, constraints, and solution approaches, we refer to \cite{prause2024approximating} which extends the overview of \cite{schlechte2023bouquet}.
In the considered preventive maintenance, service tasks are scheduled when a certain threshold for the traveled distance or the elapsed time is reached. Thus, the present degradation behavior is linear.
In all of these approaches, maintenance is incorporated into the model either through the usage of depot-to-depot paths or through the utilization of resource-constrained paths.

Next, we consider literature on the RSRP with predictive maintenance (\problemabbrv).
In \cite{herr2017predictive}, the maintenance of the vehicles is scheduled based on their health states, which represent the RUL and are considered to be point-estimates. Here, the maximal degradation level is capped by a threshold and a linear degradation behavior is assumed.
The trips are combined into predefined tours that are assigned to the vehicles and after which maintenance services can be carried out. They apply a MIP formulation to solve the problem, where the objective aims at maximizing the degradation of each vehicle before maintenance.
This approach is extended by \cite{herr2017joint} by supposing that the trips are no longer grouped into fixed tours, but it is assumed that the vehicles can be maintained at every station. Their aim is to maximize the minimum deterioration level of all vehicles before maintenance.
In \cite{wu2019train}, the vehicles have to be assigned to paths consisting of fixed tasks of the timetable. This article shows the economic impact of predictive maintenance on commuter trains in Taiwan. The authors give a MIP formulation that minimizes the expected failure costs of the vehicles and use a two-parameter Weibull distribution to model the reliability of the rolling stock.
The authors of \cite{rokhforoz2021hierarchical} assume the RUL to be a random variable, since point-estimates do not capture its uncertainty. They use normal distributions to depict this uncertainty and present a short-term approach that can be applied whenever a fault is detected during operation. In their approach, they use an exponential penalty function to enforce maintenance and employ a MIP formulation with bigM constraints. They also provide a literature overview on RUL estimation.
\cite{bougacha2022impact} aim to reinforce the problem stated by \cite{herr2017predictive}. For this purpose, they assume that the health states are uncertain and model the degradation as a gamma process. Their approach also schedules the vehicles' maintenance when the predicted health states exceed a given threshold, rather than based on the failure probability of the vehicles. They aim to optimize and study the duration of the decision horizon and propose a genetic algorithm and two heuristics for solving the problem.
In \cite{prause2024approximating}, the health states of the vehicles are considered as random variables and maintenance is scheduled based on the failure probability of the vehicles. The authors assume that these random variables are distributed by a one-parameter family of probability distribution functions (PDF). During the operation of the trips, the parameter corresponding to the health state of the vehicle in use is then updated. They present a heuristic and a lower bound for \problemabbrv\ based on a state-expanded event-graph (\graphname) in which the parameters of the health states are discretized.
Another heuristic using this graph model was presented in \cite{prause2023multi}.

In the following, we briefly review literature that utilizes graph models and algorithmic methods similar to those employed in this article.
A state-expanded version of a space-time graph modeling the assignment of duties to vehicles was used by \cite{van2017scheduling} and \cite{li2019mixed} for the scheduling of electric vehicles and their charging processes. Here, each node of the underlying space-time graph exists multiple times with different states of charge, and the energy consumption or the recharging is implicitly modeled by arcs that connect nodes with different states of charge. Thus, the energy consumption of the vehicles does not have to be considered during the solution process.
The same idea of resource-expanding the underlying graph was used by \cite{zhu2012three} to solve the resource-constrained shortest path problem.
An iterative refinement approach similar to the one used in this article was presented by \cite{boland2017continuous}. They propose a solution approach to the continuous-time service network design problem that relies on discretizing the time. The considered discretization is then iteratively refined, and the generated solutions provide increasingly accurate approximations to the original problem.
Their approach was subsequently applied to various time-dependent problems, see for example \cite{vu2020dynamic,he2022dynamic}.

\subsection{Contribution}
We consider \problemabbrv\ as defined in \cite{prause2024approximating}, where the health states of the vehicles are assumed to be uncertain and considered as random variables distributed by a parametric family of PDFs. This provides the possibility to schedule maintenance based on the failure probability of the vehicles instead of applying thresholds.
In addition, non-linear degradation functions are allowed, vehicles can be maintained only at specific maintenance locations, and the trips can be arbitrarily assigned, i.e., they are not \textit{a priori} grouped into tours.

In this article, we construct a rounding function that allows for consistent underestimation of the parameters during the construction of a \graphname\ for parameter spaces of finite dimension.
Furthermore, we introduce a family of discretizations with increasing granularity and specify the properties of the potentially eligible failure probability functions.
Using these notions, we show that the error between a vehicle rotation in the approximate problem induced by the \graphname\ and its corresponding rotation in the original \problemabbrv\ instance approaches zero for increasingly fine discretizations.
We then prove that the parameters describing the health states of the vehicles in the \graphname\ are actually underestimated compared to the parameters occurring in the original scenario when the \graphname\ is generated using the constructed rounding function. Thus, the costs of the rotations are also underestimated and the objective value of an optimal solution for the approximate problem provides a lower bound for the value of an optimal solution for \problemabbrv.
This generalizes the lower bound given in \cite{prause2024approximating} to parametric families of PDFs with more than one parameter.
Moreover, an iterative refinement of the underlying discretization of the \graphname\ leads to a sequence of solutions with increasing quality, all of which underestimate the value of an optimal solution for \problemabbrv, resulting in a dual solution approach.
Finally, the presented approaches are evaluated using a set of test instances for \problemabbrv\ given in \cite{prause2023construction}.

\subsection{Outline}
This article is structured as follows:\
First, we describe the \problemabbrv\ in Section~\ref{sec:problem}.
Then, in Section~\ref{sec:graph}, we present the \graphname, which is a state-expanded version of an event-graph.
This graph induces an approximate problem that can be solved using the ILP formulation described in Section~\ref{sec:inducedProblem}.
We then discuss in Section~\ref{sec:degradation} different approaches to model the degradation and how these models can be included in the construction of the \graphname.
Next, we construct a rounding function, which is necessary to consistently underestimate the parameters when constructing the \graphname\ in Section~\ref{sec:consistentlyRounding}.
In Section~\ref{sec:approximating}, we then utilize this function to prove that the solutions of the induced approximate problem give a lower bound for the objective of \problemabbrv.
Finally, we conduct computational experiments on a test library for \problemabbrv\ given by \cite{prause2023construction} in Section~\ref{sec:results} and close with a conclusion in Section~\ref{sec:conclusion}.

\section{Problem Formulation}
\label{sec:problem}
In this section, we recall \problemabbrv\ as described in \cite{prause2024approximating}.
Suppose we are given a set of vehicles $\vehicles$ with individual health states.
We consider the health states of the vehicles $\vehicle\in\vehicles$ at each time point $\timept\in\timepts$ as random variables $\health{\vehicle}{\timept}$ to account for their uncertainty. Here, $\timepts$ is the time horizon consisting of a finite number of time steps.
The health states are assumed to be distributed by a parametric family of probability distribution functions (PDFs) $\pdf=\{\pdf_{\param}\mid\param\in\paramspace\}$ with parameter space $\paramspace\subset\reals^n$. Examples of such parametric families are the family of normal distributions\linebreak $\mathcal{N}=\{\normal{\mu}{\variance}\mid(\mu,\variance)\in\reals\times\strictpos{\reals}\}$, the family of Poisson distributions $\pois=\{\pois(\lambda)\mid\lambda\in\strictpos{\reals}\}$, the family of gamma distributions\linebreak $\Gamma=\{\gammadist{\kappa}{\lambda}\mid (\kappa,\lambda)\in\strictpos{\reals}^2\}$, or the family of Weibull distributions\linebreak $\weibull=\{\weibull(\kappa,\lambda)\mid (\kappa,\lambda)\in\strictpos{\reals}^2\}$. Each PDF of the family is characterized by its parameters $\param\in\paramspace$, and each vehicle has an initial health value $\health{\vehicle}{0}$, which is also distributed by a PDF of the considered family, i.e., can be described by its parameter $\param_{\vehicle,0}\in\paramspace$.

Next, we consider a timetable $\trips$ consisting of trips that need to be operated. Each $\trip\in\trips$ has a departure and an arrival time $\deptime{\trip},\arrtime{\trip}\in\timepts$, and a departure and an arrival location $\deploc{\trip},\arrloc{\trip}\in\locations$, where $\locations$ is the set of locations and $\maintenances\subset\locations$ represents the maintenance facilities. In each of these workshops $\maintenance\in\maintenances$, maintenance services can be carried out to replenish the health state of the maintained vehicle and reset its associated parameters to $\param_{\maintenance}\in\paramspace$. These service actions take $\timept_{M}\in\pos{\reals}$ time.
In addition, each trip possesses a degradation function $\degradation{\trip}:\paramspace\rightarrow\paramspace$. This function determines how the parameters of the vehicles' health state alter when $\trip$ is operated.
For example, if we consider a vehicle $\vehicle$ with current health state $\health{\vehicle}{\timept}\sim \pdf_{\param_{\vehicle,\timept}}$ and let this vehicle operate trip $\trip\in\trips$, we apply $\degradation{\trip}$ to the parameters that characterize $\health{\vehicle}{\timept}$ and obtain $\param_{\vehicle,\timept+1} = \degradation{\trip}(\param_{\vehicle,\timept})$. This parameter then describes the PDF of the new health state, i.e., $\health{\vehicle}{\timept+1}\sim \pdf_{\param_{\vehicle,\timept+1}}$, at the following time point.
Assumptions regarding these functions and various degradation models are discussed in Section~\ref{sec:degradation}.
Note that we assume that there are also degradation functions associated with deadhead trips or vehicle stoppage.
Moreover, the considered family of PDFs yields the parameterized failure probability of the vehicles, denoted by $\failureprob$, which determines their potential failure costs. Examples of $\failureprob$ are presented in Section~\ref{sec:degradation} and its properties are discussed in Sections~\ref{sec:consistentlyRounding} and \ref{sec:approximating}.
Finally, $\nvehicle{\trip}\in\strictpos{\integer}$ specifies the number of vehicles required to operate $\trip$.

The task of \problemabbrv\ is then to determine an assignment of trips, deadhead trips, stoppages, and maintenance services to each of the vehicles such that all trips are operated. In addition, we associate costs with all of the aforementioned operations as well as with the breakdown of vehicles during operation and aim for an assignment with minimal total costs, including potential failure costs of the vehicles. Finally, we require that the resulting vehicle rotations are balanced, i.e., the numbers of vehicles located at each destination at the beginning and at the end of the time horizon have to coincide.

\section{The State-Expanded Event-Graph}
\label{sec:graph}
To approximate \problemabbrv, we recall the construction of the state-\linebreak expanded event-graph (\graphname) introduced in \cite{prause2024approximating}. This graph induces an approximate problem, which can be formulated and solved by an ILP. Thus, if all parameters of the original problem are consistently underestimated during the construction of the graph, the LP relaxation of the ILP provides a lower bound for the original instance.
As mentioned in Section~\ref{subsec:relatedWork}, similar graphs were used by \cite{van2017scheduling} and \cite{li2019mixed} for modeling the state of charge of a fleet of electric vehicles, and by \cite{zhu2012three} for solving the resource-constrained shortest path problem.
Note that the arc reduction procedures used in \cite{kliewer2006time} are also applied in the construction of the \graphname.

\subsection{Constructing the Nodes}
\label{subsec:nodeConstruction}
First, we describe the generation procedure of the nodes of the \graphname. Their construction is based on the events of the respective trips of the timetable, i.e., the times and locations of their departures and arrivals, as well as a discretization $\discr$ of the parameter space $\paramspace$.
\begin{example}
    Assume that the random variables representing the health states are normally distributed and take values in $[0,1]$. Then, a possible discretization of the corresponding parameter space $\paramspace = \{(\mu,\variance)\in\reals\times\strictpos{\reals}\}$ could be
    \begin{equation*}
        \discr = \big\{(\mu,\variance)\in\reals^2\mid\mu\in\{0.0,0.1,\dots,1.0\},\variance\in\{0.05,0.1,\dots,0.95\}\big\}.
    \end{equation*}
\end{example}

To generate the nodes of the \graphname, we iterate over the individual trips $\trip\in\trips$ and consider their corresponding departure and arrival events.
For each of these events and each parameter of the discretization $\discr$, we construct a node. This results in the set of departure nodes $\nodes_{+}\coloneqq\{(\deploc{\trip},\deptime{\trip},\param)\mid\linebreak\trip\in\trips,\param\in\discr\}$ and the set of arrival nodes $\nodes_{-}\coloneqq\{(\arrloc{\trip},\arrtime{\trip},\param)\mid\trip\in\trips,\param\in\discr\}$.
In addition, we add a start and an end node for each location and each parameter of the discretization, i.e., $\nodes_{0}\coloneqq\{(\location,\timept,\param)\mid\location\in\locations,\timept\in\{0,\infty\},\param\in\discr\}$.
Furthermore, we define the set of nodes corresponding to a certain location $\location\in\locations$ and a specific time $\timept\in\timepts$ as $\nodes(\location,\timept)\coloneqq\{(\location,\timept,\param)\in\nodes\mid\param\in\discr\}$.

\subsection{Constructing the Arcs}
\label{subsec:arcConstruction}
Next, we describe how to construct the arcs of the \graphname. Note that all arcs are directed and correspond to the possible operations of the vehicles. They can therefore alter the parameters that characterize the PDF of the health state of the vehicle in use. Here, the parameters of the tail characterize the PDF of the health state before the operation of the service corresponding to the considered arc, while the parameters of its head represent the updated values.
However, the exact value that results when the degradation function of a service is applied to the parameter of its tail node is not necessarily an element of $\discr$. Therefore, it may happen that no head node exists for the arc of this service.
To circumvent this problem, we will utilize a rounding function $\floorto{\cdot}{\discr}:\paramspace\rightarrow\discr$ as constructed in Section~\ref{sec:consistentlyRounding}, which maps arbitrary parameters of $\paramspace$ to elements of $\discr$.
\newpage

When generating the arcs, the employed procedure differs depending on the type of service to be represented:\
First, we construct the trip arcs. Let therefore $\trip\in\trips$ be an arbitrary trip. Then, we iterate over all corresponding departure nodes, i.e., nodes in $\nodes(\deploc{\trip},\deptime{\trip})$, and apply the degradation function of the trip, i.e., $\degradation{\trip}$, to their parameter values.
Let $\node_{1} = (\deploc{\trip},\deptime{\trip},\param_{1})$ be one of these departure nodes, then the parameter value after the operation of $\trip$ is $\param_{2} = \degradation{\trip}(\param_{1})$. However, $\param_{2}$ does not have to be contained in $\discr$. We therefore map it to a value in $\discr$ by applying the rounding function mentioned above. This results in $\param_{3} = \floorto{\param_{2}}{\discr}$, and an arc is added between $\node_{1}$ and $\node_{3}=(\arrloc{\trip},\arrtime{\trip},\param_{3})$.
If we repeat this procedure for each $\trip\in\trips$, we obtain the set of trip arcs $\arcs_{\trips}$.
Note that for each $\node\in\nodes(\deploc{\trip},\deptime{\trip})$ there is an outgoing arc corresponding to $\trip$, so we obtain multiple arcs representing $\trip$, which we denote by $\arcs(\trip)$.

Next, we add waiting arcs representing the standstill of the vehicles at their current location. For this purpose, we sort the nodes at each location $\location\in\locations$ in chronologically ascending order and connect time-consecutive nodes with an arc. This creates a timeline for each location. Again, we apply the degradation function associated with waiting to the parameters of the tail node and then determine the parameters of the head node by using $\floorto{\cdot}{\discr}$.
The degradation function used here can depend on the waiting time or simply be the identity if we assume that no wear occurs during the standstill.
Repeating this process for all locations yields the set of waiting arcs $\arcs_{W}$.

Now, we add arcs that represent possible deadhead trips between the different locations.
Suppose we are given a node $\node_1=(\location_1,\timept_1,\param_1)\in\nodes$, then we define the function \texttt{firstAfter($\node_1,\location_2$)} to determine the first departure node at $\location_2$ that can be reached by a deadhead trip from $\location_1$ at time $\timept_1$. Analogously, we define \texttt{lastBefore($\node_2,\location_1$)} to return the chronologically last arrival node at $\location_1$ from which it is possible to reach $\node_2$ by a deadhead trip.
To obtain all deadhead arcs, we now iterate over all arrival nodes $\node_1=(\location_1,\timept_1,\param_1)\in\nodes_{-}$ and determine $\node_2=\texttt{firstAfter(}\node_1,\location_2\texttt{)}$ for all $\location_2\in\locations\setminus\{\location_1\}$. If $\node_1=\texttt{lastBefore(}\node_2,\location_1\texttt{)}$ holds, we add an arc from $\node_1$ to $\node_2$, otherwise we omit the addition of the arc. This generates the set of deadhead arcs $\arcs_D$.
In this process, the composition of the rounding function with the degradation function associated with deadhead trips is again applied to determine the parameters of the head nodes of the arcs. The degradation function can be time- or distance-dependent, or the identity function if it is assumed that no degradation occurs while deadheading.
Note that the utilization of \texttt{firstAfter} and \texttt{lastBefore} reduces the number of deadhead arcs contained in the \graphname\ and was introduced in \cite{kliewer2006time}.

Then, we add maintenance arcs. These arcs are used to restore the health states by resetting the parameters to predefined values that are as good as new.
Let $\param_M\in\paramspace$ be the parameter values after maintenance, then we define $\param_m\coloneqq\floorto{\param_M}{\discr}\in\discr$ and iterate over all events $(\location,\timept)\in\locations\times\timepts$ and maintenance locations $\maintenance\in\maintenances$. For each of these combinations, we add a new node $\node_{\maintenance} = (\maintenance,\timept + \timept_{\location,\maintenance} + \timept_{M},\param_\maintenance)$ representing the corresponding maintenance service. Here, $\timept_{\location,\maintenance}$ is the travel time from $\location$ to $\maintenance$ and $\timept_{M}$ is the duration of the service. We then add an arc from all nodes of $\nodes(\location,\timept)$ to the just created maintenance node $\node_{\maintenance}$.
Finally, we add the outgoing deadhead and waiting arcs of $\node_{\maintenance}$ as described above.
These procedures generate the set of maintenance nodes $\nodes_M$ and the set of maintenance arcs $\arcs_M$.

Lastly, we add artificial nodes and arcs that serve to integrate the initial health of the vehicles and their origin into the graph and that are later used to guarantee the balancedness of the vehicle rotations.
For this purpose, we add a node $\node_{\location}$ for any origin where a vehicle is initially located. Then, for each of these vehicles, we add an arc from $\node_{\location}$ to $(\location,0,\floorto{\param_{\vehicle,0}}{\discr})$, where $\param_{\vehicle,0}$ is the parameter value of the initial health state of $\vehicle$. This results in the set of artificial starting arcs $\arcs_{\artificial}^{0}$.
Similarly, we add a node for each location $\location\in\locations$ symbolizing the end of the time horizon and add an arc from each node of $\nodes(\location,\infty)$ to it. We denote these artificial end arcs by $\arcs_{\artificial}^{\infty}$ and thus obtain the artificial node set $\nodes_{\artificial}$ and the artificial arc set $\arcs_{\artificial}\coloneqq\arcs_{\artificial}^{0}\cup\arcs_{\artificial}^{\infty}$.

\subsection{The Resulting Graph}
Using the node and arcs sets constructed in the previous Sections~\ref{subsec:nodeConstruction} and \ref{subsec:arcConstruction}, we define $\nodes\coloneqq\nodes_{+}\cup\nodes_{-}\cup\nodes_{0}\cup\nodes_{M}\cup\nodes_{\artificial}$ and $\arcs\coloneqq\arcs_{\trips}\cup\arcs_{W}\cup\arcs_{D}\cup\arcs_{M}\cup\arcs_{\artificial}$, which together form the \graphname\ $\graph=(\nodes,\arcs)$.
Note that it is admissible to repeatedly remove all nodes of $\nodes\setminus\nodes_{\artificial}$, together with their incident arcs, that have either in- or outdegree equal to zero.
In addition, $\graph$ is always dependent on the choice of the discretization $\discr$ and the rounding function applied during the construction process.

We would like to emphasize that non-linear degradation functions can be used in the construction of the \graphname, which offers the possibility to consider and incorporate a wider variety of degradation behaviors. This is an advantage over classical space-time graph models, where it is complicated to deal with non-linear resource consumption.
Additionally, the \graphname\ can easily be adapted for distance- or time-based preventive maintenance scenarios. Here, the parameter values are chosen to represent the distance traveled or the elapsed time since the last maintenance. Then, the bounds of $\paramspace$ must to be set to match the thresholds of the maximum allowed mileage and the degradation functions are set to accumulate the distance or time associated with the different operations.
Possible extensions of the presented model and how these can be implemented is discussed in \cite{prause2024approximating}.

\subsection{Arc Costs}
Now, we describe the costs of the arcs of the \graphname. These costs depend on the type of the considered arc.
Waiting arcs ($\arcs_{W}$) are assigned costs of zero, while the costs of the deadhead arcs ($\arcs_{D}$) depend on the traveled distance between their start and end locations.
The costs of the artificial arcs ($\arcs_{\artificial}$) are determined in the same way as the costs of the deadhead trips, but the operation costs of a vehicle for the considered time horizon are additionally added to each of the starting arcs ($\arcs_{\artificial}^{0}$).
The two types of maintenance arcs ($\arcs_{M}$), i.e., the arcs heading to a maintenance node and the arcs originating from it, are allocated different costs. The incoming ones are assigned the sum of the maintenance costs and the deadhead costs between the start and the end location of the arc, while the outgoing arcs are only attributed the deadhead costs from the maintenance facility to the arrival location.

The trip costs ($\arcs_{\trips}$) also consist of two different sub-costs. First, there are the costs that are inherent to every trip and which can be time- or distance-dependent. Second, we assign them the expected failure costs of the vehicles, i.e., the product of the failure costs and the probability that a vehicle breaks down.
Various approaches to model degradation and their corresponding failure probability functions ($\failureprob$) are discussed in Section~\ref{sec:degradation}.
These functions are then applied to the parameters characterizing the health state of the vehicle during and after the operation of a trip, i.e., to the parameter values of the arc's head node.
This can be exemplified as follows:\ A trip arc\linebreak $\arc = (\node_1,\node_2)$ connecting $\node_1 = (\location_1,\timept_1,\param_1)$ and $\node_2 = (\location_2,\timept_2,\param_2)$ corresponds to a vehicle whose original health state is a random variable distributed by PDF $\pdf_{\param_1}$, which is changed to a health state that is a random variable distributed by $\Pi_{\param_2}$ by performing the associated trip.
Thus, the parameter values of the head node of a trip characterize the failure probability of a vehicle during the trip. The expected failure costs can therefore be determined by multiplying the costs associated with a failure with $\failureprob(\param_2)$.
Note that the \graphname\ has the advantage of determining the failure costs of the trips during its construction and thus the failure probability function does not need to be evaluated in the objective, as in the case of non-state-expanded MIP formulations.

\subsection{The Completely Expanded Event-Graph}
Next, we describe a special case of the \graphname\ that not only approximates but exactly models \problemabbrv.
For this purpose, we will construct a discretization $\discr$ that contains all possibly occurring parameters.
First, we collect all parameter values corresponding to the initial health states of the vehicles in $\discr$. Then, we sort the trips in ascending order by their departure times, iterate over all of them, and apply each degradation functions to all values currently contained in $\discr$. The resulting values are in turn added to $\discr$ and the procedure continues with the next trip.
Note that $\abs{\discr}$ may be exponential in the input size of \problemabbrv. If we now apply the identity $\identity$ as rounding function, we obtain $\floorto{\param}{\discr} = \identity(\param) = \param\in\discr$, since $\discr$ contains every possible parameter value that could occur. Therefore, the paths in the \graphname\ corresponding to the vehicle rotations have the same parameter values as they would have in the continuous case. Thus, a solution to the flow problem induced by this graph is an optimal solution for \problemabbrv.
We refer to this graph as completely expanded event-graph (\completegraphname). Note that this graph is only of theoretical importance and used for comparisons with the \graphname\ and not for computations.

\section{The Induced Approximate Flow Problem}
\label{sec:inducedProblem}
Suppose we consider an instance of \problemabbrv, and we constructed the \graphname\ $\graph$ based on some discretization $\discr$, then we need to translate the task of the original problem into a problem in $\graph$.
Since paths in $\graph$ that start at one of the artificial start nodes and head to one of the artificial end nodes correspond to possible assignments of tasks to a vehicle, we need to determine a set of paths with the following properties:\
First, each trip $\trip\in\trips$ must be operated, i.e., exactly $\nvehicle{\trip}$ arcs out of the multiple ones corresponding to $\trip$ have to be selected.
Next, each artificial arc representing the usage of a vehicle, i.e., arcs in $\arcs_{\artificial}^{0}$, can be contained in at most one path.
Finally, to ensure the balancedness, the number of paths leaving a location $\location\in\locations$ must be equal to the number of paths heading to $\location$.

Note that optimal solutions to this problem are only approximations to the original \problemabbrv, since the parameters are rounded during the construction of the \graphname. Thus, an approximation error is introduced to each of the parameters and the associated failure probability of the vehicles is not exactly determined, making the costs of the trip arcs only approximations.
However, with finer approximations, i.e., more granular choices of $\discr$, this error approaches zero, as we show in Section~\ref{sec:approximating}.

\subsection{Complexity of the Approximate Problem}
After describing the approximate problem, we give a mathematical formulation for it and prove that it is NP-hard.
\begin{definition}[EPCP]
    \label{def:epcp}
    Given a directed graph $\graph=(\nodes,\arcs)$ with source and sink nodes $\{\varsi{1},\dots,\varsi{L},\varti{1},\dots,\varti{L}\}\subset\nodes$ and an index set $I$. For each $i\in I$, let $\arcs_i\subseteq\arcs$ be a subset of arcs and let $c:\arcs\rightarrow\pos{\reals}$ be a cost function associated with the arcs of $\graph$. Then the task of the \emph{exact path cover problem (EPCP)} is to find a set of paths $\paths$ with minimum costs, each of which starts at some source node $\varsi{k}$ and heads to some sink node $\varti{l}$, and which together cover exactly one arc of each set $\arcs_i$, $\forall i\in I$.
    In addition, $|\outgoing(\varsi{l})|=|\incoming(\varti{l})|$ must hold for all $l\in\{1,\dots,L\}$.
\end{definition}
The solution of the EPCP for a given \graphname\ yields a solution to the induced approximate problem. Here, trips with $\nvehicle{\trip}>1$ must be modeled by $\nvehicle{\trip}$ parallel trips.
After defining EPCP, let us recall \emph{exact cover}, a problem stated by \cite{karp1972reducibility}, which was proven to be NP-complete by \cite{garey1979computers}:\
Given a set of elements $X$ and collection $C$ of subsets of $X$, i.e., $C\subseteq\powerset{X}$. Then the task of \emph{exact cover} is to find a subcollection $S\subseteq C$ whose disjoint union is $X$, i.e.,
\begin{equation*}
    \disjointUnion\limits_{Y\in S} Y = X.
\end{equation*}

If we now have an additional weight function $w:C\rightarrow \pos{\reals}$ that assigns a weight to each subset $Y\in C$, we obtain the \emph{weighted exact cover} problem, where the task is to find an exact cover with minimum total weight.
This weighted variant is the optimization version of exact cover and therefore NP-hard.

\begin{figure}
    \centering
    \input{reductionGraph}
    \caption{Graphical representation of an exact cover instance as EPCP. The set of elements is $X=\{a,b,c,d,e,f,g\}$ and the collection of possible subsets is $C=\big\{\{a,d,g\},\{a,d\},\{d,e,g\},\{c,e,f\},\{b,c,f,g\},\{b,g\}\big\}$.}
    \label{fig:reductionGraphExample}
\end{figure}
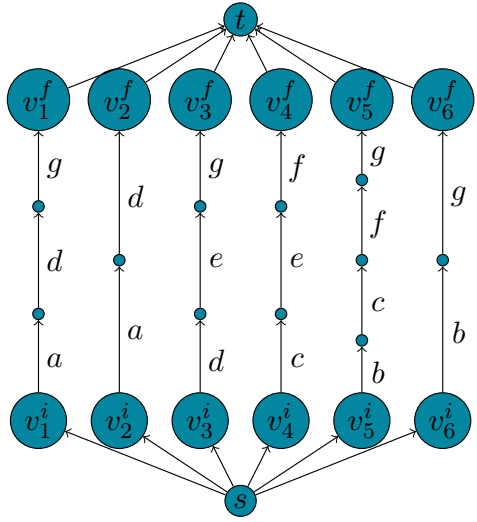

\begin{theorem}
    \label{thm:nphardness}
    EPCP is NP-hard.
\end{theorem}
\begin{proof}
    To prove the NP-hardness of EPCP, we give a reduction of weighted exact cover to it.
    Suppose we are given an instance of weighted exact cover with elements $X$, collection $C\subseteq\powerset{X}$, and weight function $w$ over $C$.
    We will now set up the graph $\graph=(\nodes,\arcs)$ of the associated EPCP instance. An illustration of the graph resulting from the constructed described below is depicted in Figure~\ref{fig:reductionGraphExample}.
    First, we introduce a source node $s$ and a sink node $t$ and set $\nodes\coloneqq\{s,t\}$.
    Next, we iterate over each subset $Y_k\in C$ and add an initial node $\node_{k,i}$ and a final node $\node_{k,f}$ to $\nodes$. Between these nodes, we insert chains of arcs corresponding to the elements of $Y_k$, where the first arc has $\node_{k,i}$ as its tail, while the last arc has $\node_{k,f}$ as its head. For each $e\in X$, all arcs corresponding to $e$ are collected in set $\arcs_e$, and we add each generated arc to $\arcs$.
    Then, we then add an arc from $\vars$ to each of the initial nodes $\node_{k,i}$ and an arc from each of the final nodes $\node_{k,f}$ to $\vart$.
    We set the costs of all arcs to zero, except for the costs of the arcs between $\node_{k,f}$ and $\vart$, which are assigned the weight of the corresponding set $Y_k$.
    Notice that the balancedness of the paths, which is explicitly required in Definition~\ref{def:epcp}, is trivially satisfied since we consider only one source and one sink node.
    A solution for EPCP for the just generated graph $\graph=(\nodes,\arcs)$ and w.r.t.\ the arc sets $A_e$ then provides a solution for weighted exact cover.
\end{proof}

Note that even if EPCP could be solved in polynomial time, the underlying graph, i.e., the \completegraphname, which represents the original \problemabbrv\ instance, would have an exponential size.

\subsection{Solving the Approximate Problem}
\label{subsec:solving}
After introducing the approximate problem induced by the \graphname, we reproduce the following ILP formulation (\hyperref[eq:objective]{AP}) to solve this problem.
The presented formulation was given in \cite{prause2024approximating}, where it was applied to approximate \problemabbrv\ with one-dimensional parameter space. But it can also be directly employed for the multi-dimensional case if the \graphname\ is constructed accordingly.

In the ILP, $\cost_\arc\in\reals$ are the costs of the arcs and $\varx_\arc\in\pos{\integer}$ are variables that determine how many vehicles traverse arc $\arc\in\arcs$.
$\outgoing(\location_0)$ and $\incoming(\location_\infty)$, for $\location\in\locations$, are the outgoing arcs of the artificial start node at $\location$ and the incoming arcs of the artificial end node at $\location$, respectively.
In addition, $\nodes_{\artificial}$, $\arcs_{\artificial}$, and $\arcs_{\artificial}^0$ are as defined in Section~\ref{sec:graph}.

\begingroup
\begin{alignat}{8}
\text{(AP)} \quad && \min && \sum\limits_{\arc\in\arcs} \cost_\arc \varx_\arc \label{eq:objective} \\
&& \text{s.t.} && \sum\limits_{\arc\in\arcs(\trip)} \varx_\arc &= \nvehicle{\trip} && \forall \trip \in \trips \label{eq:coverage} \\
&&&& \sum\limits_{\arc\in\incoming(\node)} \varx_\arc &= \sum\limits_{\arc\in\outgoing(\node)} \varx_\arc && \forall \node \in \nodes\setminus\nodes_{\artificial} \label{eq:flow_conservation} \\
&&&& \sum\limits_{\arc\in\outgoing(\location_{0})} \varx_\arc &= \sum\limits_{\arc\in\incoming(\location_{\infty})} \varx_\arc \quad && \forall \location \in \locations \label{eq:balancedness} \\
&&&& \varx_\arc &\in \pos{\integer} && \forall \arc \in \arcs\setminus\arcs_{\artificial}^{0} \label{eq:arc_variable} \\
&&&& \varx_\arc &\in \binary && \forall \arc \in \arcs_{\artificial}^{0}. \label{eq:start_arc_variable}
\end{alignat}
\endgroup
Here, the objective function~\eqref{eq:objective} aims to minimize the total costs of the vehicle rotations. Constraints~\eqref{eq:coverage} ensure that each trip is operated with the required number of vehicles. Flow conservation is guaranteed by constraints~\eqref{eq:flow_conservation}, while constraints~\eqref{eq:balancedness} balance the number of vehicles at each location. Finally, the variables representing the selected arcs of the \graphname\ and their domains are defined in \eqref{eq:arc_variable} and \eqref{eq:start_arc_variable}, where the domain of the artificial starting arcs is binary, since each vehicle can only be used once.
Possible extensions of this model are also discussed in \cite{prause2024approximating}.

\section{Modeling the Health States and Degradation}
\label{sec:degradation}
In this section, we present two different models that interpret the health states of the vehicles as random variables. These models are crucial as they imply the associated failure probability function $\failureprob$.

\subsection{Remaining Useful Life Approach}
The first model relies on the idea that the RUL of a vehicle can be represented as a point estimate of an observable quantity that is subject to some uncertainty. Examples of this would be a train bogie with an estimated RUL of 10,000 km $\pm$ 500 km or a door that can perform 1,500 $\pm$ 150 opening-closing-cycles before a failure occurs.
These interpretations reflect the representation via a mean value and a confidence interval and could therefore be represented by the family of normal distributions $\{\normal{\mu}{\variance}\mid\mu\in\reals,\variance\in\strictpos{\reals}\}$.
By scaling the considered quantity, it can be assumed that the health states are random variables taking values in $[0,1]$, where 1 denotes a state that is as good as new and 0 represents a state that leads to a failure. Here, it would be reasonable to assume that the application of the degradation functions of the services decreases $\mu$ as the vehicle deteriorates, while the uncertainty about the state, i.e., $\variance$, increases with each prediction.
This model would then lead to the following failure probability function:\
\begin{equation*}
    \failureprob(\param) \coloneqq \prob[\vehicle\text{ has a failure}] = \prob[\health{\vehicle}{\timept} \leq 0] = \int\limits_{-\infty}^{0} \pdf_{\param}(x)\,dx,
\end{equation*}
where $\health{\vehicle}{\timept}\sim\normal{\mu}{\variance}$ is the random variable representing the health state of vehicle $\vehicle$ at time $\timept$, and $\pdf_{\param}(x)$ is the PDF of a normal distribution with parameters $\param = (\mu,\variance)^T$ characterizing $\health{\vehicle}{\timept}$.
Note that the failure probability function $\failureprob$ depends solely on the parameters $\param\in\paramspace$.
A visualization of the failure probability of the family of normal distributions as a function of $\mu$ and $\variance$ is shown in Figure~\ref{fig:failureProbExample}.

\begin{figure}
    \centering
    \begin{subfigure}{0.45\linewidth}
        \includegraphics[width=\linewidth]{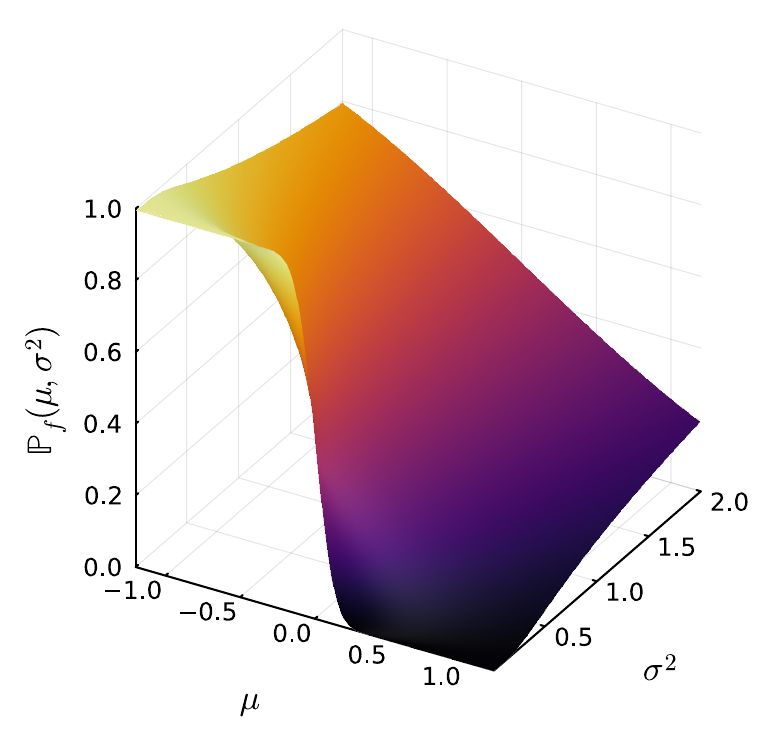}
        \caption{Failure probability $\failureprob$ of the family of normal distributions as a function of $\mu$ and $\variance$.}
        \label{fig:failureProbExample}
    \end{subfigure}
    \hfill
    \begin{subfigure}{0.45\linewidth}
        \includegraphics[width=\linewidth]{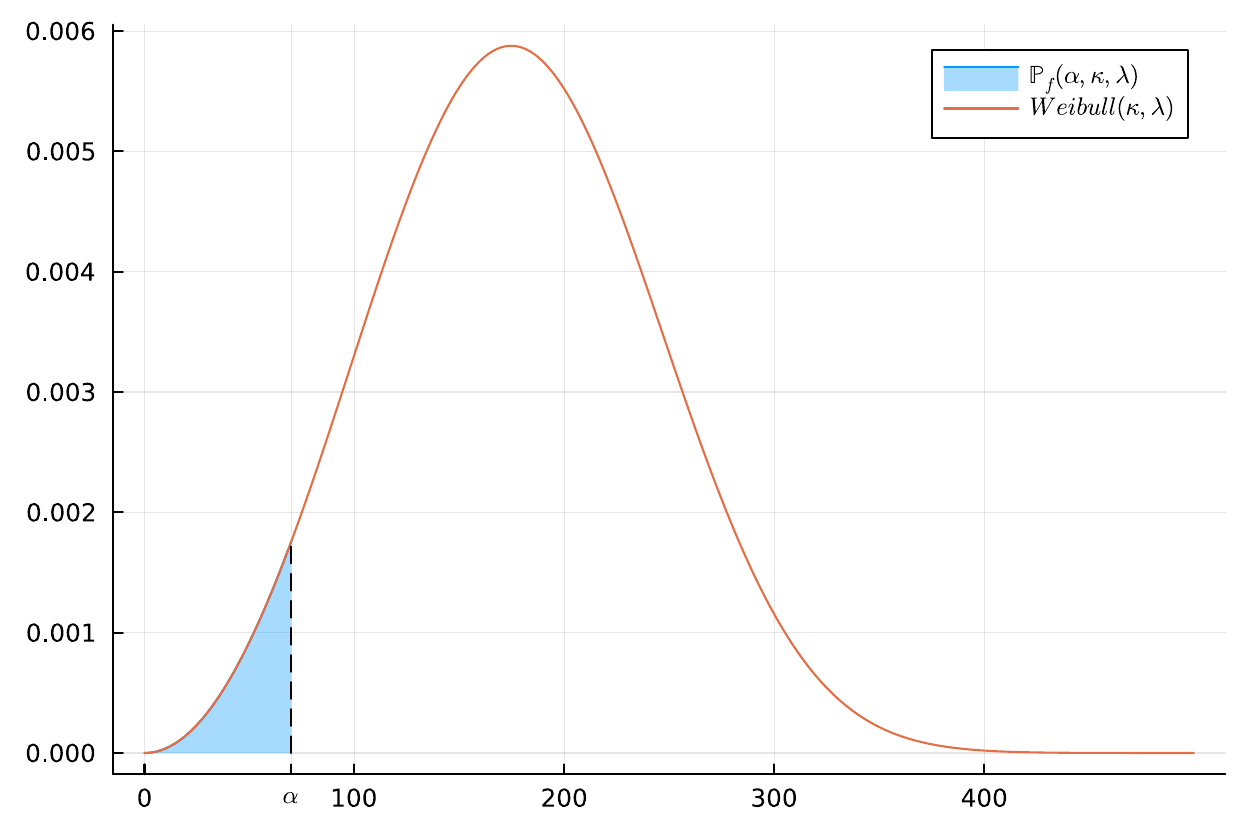}
        \vspace{0.5cm}
        \caption{PDF of a Weibull distribution with $\kappa=3$ and $\lambda=200$. The failure probability $\failureprob$, for $\alpha = 70$, is shaded.}
        \label{fig:weibullDistExample}
    \end{subfigure}

    \begin{subfigure}{0.45\linewidth}
        \includegraphics[width=\linewidth]{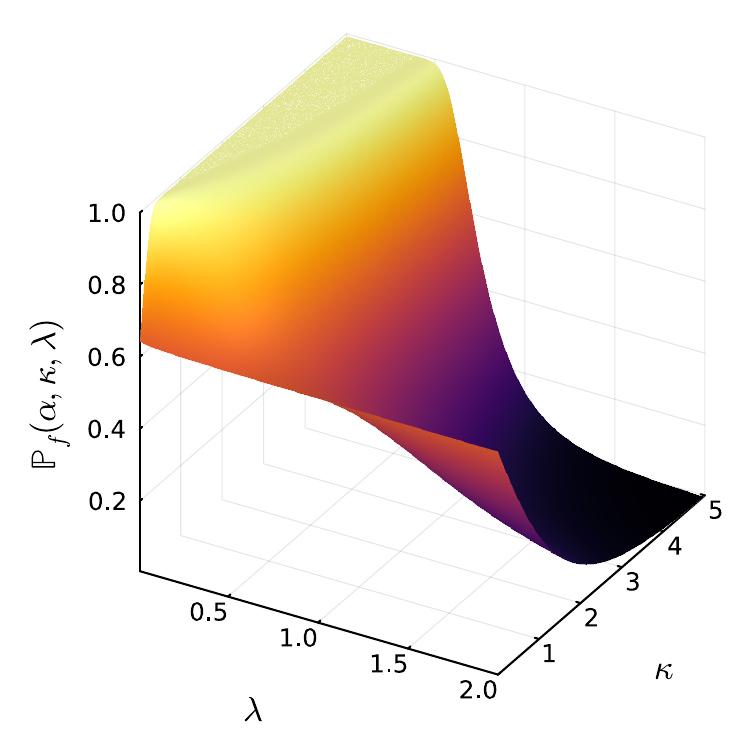}
        \caption{Failure probability $\failureprob$ of the family of Weibull distributions as a function of $\kappa$ and $\lambda$, with fixed $\alpha = 0.7$.}
        \label{fig:failureProbExample2}
    \end{subfigure}
    \hfill
    \begin{subfigure}{0.45\linewidth}
        \includegraphics[width=\linewidth]{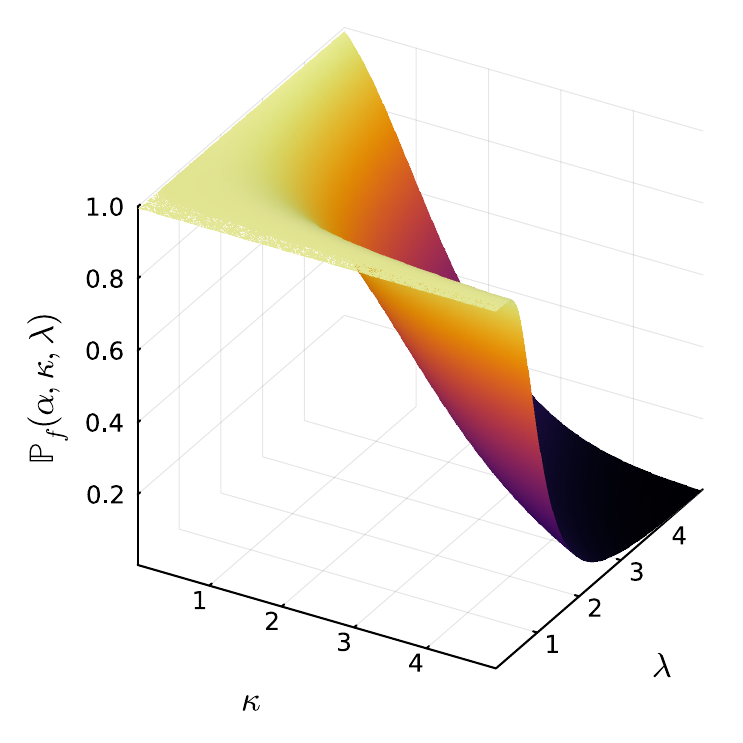}
        \caption{Failure probability $\failureprob$ of the family of gamma distributions as a function of $\kappa$ and $\lambda$, with fixed $\alpha = 5$.}
        \label{fig:failureProbExample3}
    \end{subfigure}
    \caption{Failure probability functions for different families of PDFs.}
\end{figure}

\subsection{Reliability Approach}

Another possible interpretation is based on reliability theory. In this context, the health state represents the failure rate of the corresponding vehicle over the remaining useful life and depicts the probability that the vehicle fails after an additional number of operating cycles or a certain mileage in its current condition.
For an introduction to reliability theory, see \cite{emmert2019introduction}.
Two distributions that are frequently used in this context are the Weibull and gamma distributions.
If we use such a representation, we need to take into account the additional mileage of the operated services, i.e., how long the traveled distances are or how many opening-closing-cycles occur. Let this mileage be $\alpha\in\strictpos{\reals}$, then the failure probability can be determined as follows:\
\begin{equation*}
    \failureprob(\alpha,\param) \coloneqq \prob[\vehicle\text{ has a failure}] = \prob[\health{\vehicle}{\timept} \leq \alpha] = \int\limits_{-\infty}^{\alpha} \pdf_{\param}(x)\,dx,
\end{equation*}
where again $\health{\vehicle}{\timept}\sim \weibull(\kappa,\lambda)$ is the random variable representing the health state of vehicle $\vehicle$ at time $\timept$, and $\pdf_{\param}(x)$ is the PDF of a Weibull distribution with parameters $\param = (\kappa,\lambda)^T$ that characterize $\health{\vehicle}{\timept}$. In addition, $\alpha$ is, as already mentioned, the additional mileage of the operated service.
An example of a two-parameter Weibull distribution and the corresponding failure probability for a given $\alpha$ is shown in Figure~\ref{fig:weibullDistExample}.
In addition, Figure~\ref{fig:failureProbExample2} illustrates the failure probability of a Weibull distribution as a function of $\kappa$ and $\lambda$, and a fixed $\alpha$.
Note that a similar model would also be possible with gamma distributions. An example of the failure probability associated with a two-parameter gamma distribution as a function of $\kappa$ and $\lambda$, for a fixed $\alpha$, is depicted in Figure~\ref{fig:failureProbExample3}.

\section{Consistently Rounding the Parameters}
\label{sec:consistentlyRounding}
In this section, we construct a rounding function whose application allows a consistent underestimation of the parameters when generating the \graphname.
This enables us to assess the rounding error and to show the convergence of a subdivision approach which is presented in the following Section~\ref{sec:approximating}.
It can thus be shown that solutions of the EPCP for a \graphname\ constructed using this rounding function provide a lower bound to \problemabbrv.

In the following, $\diffbar{\reals^n}$ denotes the set of all continuously differentiable functions with domain $\reals^n$ and $\nabla f$ is the gradient of $f$.
Furthermore, $\conv{R}$ is the convex hull of the vectors contained in $R\subseteq\reals^n$, $\langle\cdot,\cdot\rangle$ is the standard scalar product of $\reals^n$ and $\ones{n}=(1,\dots,1)^T\in\reals^n$.

\begin{remark}
    Let $f:\reals^n\rightarrow\reals$ be in $\diffbar{\reals^n}$ and $\varx,\vard\in\reals^n$, then the \emph{directional derivative of $f$ w.r.t.\ $\vard$ at $\varx$} is given by
    $\nabla_{\vard}f(\varx)=\langle\nabla f(\varx),\vard\rangle$.
\end{remark}

\begin{definition}
    Let $f:\reals^n\rightarrow\reals$ be in $\diffbar{\reals^n}$, $\domain\subseteq\reals^n$ and $\directions\subset\reals^n$ finite. Then we call $f$ \emph{directional monotonically increasing on $\domain$ w.r.t.\ $\directions$} if we have $\nabla_{\vard}f(\varx)\geq 0$ for all $\varx\in\domain$ and $\vard\in\directions$.
\end{definition}

\begin{lemma}
    Let $f:\reals^n\rightarrow\reals$ be in $\diffbar{\reals^n}$, $\domain\subseteq\reals^n$ and $\directions\subset\reals^n$ finite. Furthermore, let $f$ be directional monotonically increasing on $\domain$ w.r.t.\ $\directions$. Then $f$ is directional monotonically increasing on $\domain$ w.r.t.\ $\conv{\directions}$.
\end{lemma}
\begin{proof}
    Since $f\in\diffbar{\reals^n}$, the derivative in any direction $\vard\in\reals^n$ can be expressed as $\nabla_{\vard}f(\varx)=\langle\nabla f(\varx),\vard\rangle$. W.l.o.g.\ let $\directions=\{r_1,\dots,r_m\}$, then for any $\vard\in\conv{\directions}$ there exists $\lambda\in\pos{\reals}^m$ with $\langle\lambda,\ones{m}\rangle=1$ such that $\vard=\lambda_1 r_1+\dots+\lambda_m r_m$.
    Since $f$ is directional monotonically increasing on $\domain$ w.r.t.\ $\directions$, we can conclude
    \begin{equation*}
        \nabla_{\vard}f(\varx)=\langle\nabla f(\varx),\vard\rangle=
        \lambda_1\underbrace{\langle\nabla f(\varx),r_1\rangle}_\text{$\geq 0$}+\dots+\lambda_m\underbrace{\langle\nabla f(\varx),r_m\rangle}_\text{$\geq 0$}\geq 0\quad\forall \varx\in\domain.\qedhere
    \end{equation*}
\end{proof}


\begin{example}
    \label{ex:directionalNormal}
    Suppose we use the family of normal distributions to represent the health states, i.e., $\{\normal{\mu}{\variance}\mid\mu\in\reals,\variance\in\strictpos{\reals}\}$. Then, the failure probability of vehicle $\vehicle$ at time $\timept$ is
    \begin{equation}
        \failureprob(\mu,\variance)=\prob[\health{\vehicle}{k}\leq 0]=\frac{1}{2}\left(1+\erf\left(\frac{-\mu}{\sqrt{2\variance}}\right)\right),
        \label{eq:failureProbNormal}
    \end{equation}
    where $\erf$ is the Gauss error function. Thus, the gradient is
    \begin{equation*}
        \nabla\failureprob(\mu,\variance)=
        \begin{pmatrix}
            -\frac{\exp\left(\frac{-\mu^2}{2\variance}\right)}{\sqrt{2\pi\variance}} \\
            \frac{\mu\sigma\exp\left(\frac{-\mu^2}{2\variance}\right)}{\sqrt{2\pi\sigma^7}}
        \end{pmatrix}.
    \end{equation*}
    Hence, we obtain
    \begin{align*}
        &\nabla_{-\sbv{1}}\failureprob=\langle\nabla\failureprob,-\sbv{1}\rangle=-\langle\nabla\failureprob,\sbv{1}\rangle=\frac{\exp\left(\frac{-\mu^2}{2\variance}\right)}{\sqrt{2\pi\variance}}>0\quad\forall\mu\in\reals,\variance\in\strictpos{\reals} \\
        &\nabla_{\sbv{2}}\failureprob=\langle\nabla\failureprob,\sbv{2}\rangle=\frac{\mu\sigma\exp\left(\frac{-\mu^2}{2\variance}\right)}{\sqrt{2\pi\sigma^7}}.
    \end{align*}
    This yields
    \begin{equation*}
        \nabla_{\sbv{2}}\failureprob\geq 0\quad\forall\mu\geq 0,\variance\in\strictpos{\reals}\quad\land\quad\nabla_{-\sbv{2}}\failureprob\geq 0\quad\forall\mu\leq 0,\variance\in\strictpos{\reals}
    \end{equation*}
    and we can conclude that $\failureprob$ is directional monotonically increasing on $\{(\mu,\variance)\in\reals^2\mid\mu\geq 0,\variance>0\}$ w.r.t.\ $\{-\sbv{1},\sbv{2}\}$, and directional monotonically increasing on $\{(\mu,\variance)\in\reals^2\mid\mu\leq 0,\variance>0\}$ w.r.t.\ $\{-\sbv{1},-\sbv{2}\}$.
\end{example}

\begin{example}
    \label{ex:directionalWeibull}
    Suppose we use the family of Weibull distributions to represent the health states, i.e., $\{\weibull(\kappa,\lambda)\mid (\kappa,\lambda)\in\strictpos{\reals}^2\}$. Then, the failure probability of vehicle $\vehicle$ at time $\timept$ operating a service with wear $\alpha$ is
    \begin{equation}
        \failureprob(\alpha,\kappa,\lambda)=\prob[\health{\vehicle}{\timept}\leq \alpha]=1-\exp\left({-\left(\frac{\alpha}{\lambda}\right)^{\kappa}}\right).
        \label{eq:failureProbWeibull}
    \end{equation}
    Thus, the gradient is
    \begin{equation*}
        \nabla\failureprob(\alpha,\kappa,\lambda)=
        \begin{pmatrix}
            \frac{\alpha^{\kappa}\kappa\exp\left({-\left(\frac{\alpha}{\lambda}\right)^{\kappa}}\right)}{\lambda^{\kappa+1}} \\
            \frac{-\alpha^{\kappa}\exp\left({-\left(\frac{\alpha}{\lambda}\right)^{\kappa}}\right)\log\left(\frac{\alpha}{\lambda}\right)}{\lambda^{\kappa}}
        \end{pmatrix}
    \end{equation*}
    and the directional derivatives are
    \begin{align*}
        &\nabla_{\sbv{1}}\failureprob=\langle\nabla\failureprob,\sbv{1}\rangle=\frac{\alpha^{\kappa}\kappa\exp\left({-\left(\frac{\alpha}{\lambda}\right)^{\kappa}}\right)}{\lambda^{\kappa+1}}\geq 0\quad\forall(\kappa,\lambda)\in\strictpos{\reals}^2 \\
        &\nabla_{\sbv{2}}\failureprob=\langle\nabla\failureprob,\sbv{2}\rangle=\frac{-\alpha^{\kappa}\exp\left({-\left(\frac{\alpha}{\lambda}\right)^{\kappa}}\right)\log\left(\frac{\alpha}{\lambda}\right)}{\lambda^{\kappa}}.
    \end{align*}
    Thus, we obtain
    \begin{equation*}
        \nabla_{\sbv{2}}\failureprob\geq 0\quad\forall \kappa\in\strictpos{\reals},\lambda\geq\alpha\quad\land\quad\nabla_{-\sbv{2}}\failureprob\geq 0\quad\forall \kappa\in\strictpos{\reals},\lambda\leq\alpha.
    \end{equation*}
    Hence, we can conclude that $\failureprob$ is directional monotonically increasing on $\{(\kappa,\lambda)\in\strictpos{\reals}^2\mid\lambda\geq\alpha\}$ w.r.t.\ $\{\sbv{1},\sbv{2}\}$, and directional monotonically increasing on $\{(\kappa,\lambda)\in\strictpos{\reals}^2\mid\lambda\leq\alpha\}$ w.r.t.\ $\{\sbv{1},-\sbv{2}\}$.
\end{example}

\begin{example}
    \label{ex:directionalGamma}
    Suppose we use the family of gamma distributions to represent the health states, i.e., $\{\gammadist{\kappa}{\lambda}\mid(\kappa,\lambda)\in\strictpos{\reals}^2\}$. Then, the failure probability of vehicle $\vehicle$ at time $\timept$ operating a service with wear $\alpha$ is
    \begin{equation*}
        \failureprob(\alpha,\kappa,\lambda)=\prob[\health{\vehicle}{\timept}\leq \alpha]=\frac{\gamma\left(\kappa,\frac{\alpha}{\lambda}\right)}{\Gamma(\kappa)}
    \end{equation*}
    where $\Gamma(\kappa)=\int_0^\infty t^{\kappa-1} \exp(-t) \,dt$ is the gamma function and $\gamma$ is the lower incomplete gamma function defined as $\gamma\left(\kappa,\frac{\alpha}{\lambda}\right)=\int_0^\frac{\alpha}{\lambda} t^{\kappa-1} \exp(-t) \,dt$.
    Thus, we obtain
    \begin{align*}
        \nabla_{\sbv{2}}\failureprob(\alpha,\kappa,\lambda)&=\langle\nabla\failureprob(\alpha,\kappa,\lambda),\sbv{2}\rangle=\deriv{\lambda}\failureprob(\alpha,\kappa,\lambda)\\
        &=-\frac{1}{\lambda\Gamma(\kappa)}\left(\frac{\alpha}{\lambda}\right)^\kappa\exp\left(-\frac{\alpha}{\lambda}\right)<0.
    \end{align*}
    Unfortunately, we were not able to obtain analytical results showing\linebreak $\nabla_{\sbv{1}}\failureprob(\alpha,\kappa,\lambda)<0$. But computational experiments indicate exactly this behavior.
    Hence, we assume $\failureprob$ to be directional monotonically increasing on $\strictpos{\reals}^2$ w.r.t.\ $\{-\sbv{1},-\sbv{2}\}$.
\end{example}

\begin{lemma}
    \label{lem:directionalExistance}
    Let $f:\reals^n\rightarrow\reals$ be in $\diffbar{\reals^n}$ and $\varx\in\reals^n$, then there exists a vector of signs $\vars=(\varsi{1},\dots,\varsi{n})^T\in\{-1,1\}^n$ such that $f$ is directional monotonically increasing on $\{\varx\}$ w.r.t.\ $\directions=\{\varsi{1}\sbv{1},\dots,\varsi{n}\sbv{n}\}$.
\end{lemma}
\newpage
\begin{proof}
    Consider $\sbv{i}$ for any $i\in\setft{1}{n}$, then we either have $\nabla_{\sbv{i}}f(\varx)\geq 0$ or $\nabla_{\sbv{i}}f(\varx)<0$.
    \begin{description}
        \item[Case 1:] If $\nabla_{\sbv{i}}f(\varx)\geq 0$, setting $\varsi{i}=1$ leads to a direction $\varsi{i}\sbv{i}$ in which $f$ is directional monotonically increasing.
        \item[Case 2:] If $\nabla_{\sbv{i}}f(\varx)<0$, setting $\varsi{i}=-1$ leads to a direction $\varsi{i}\sbv{i}$ with $\nabla_{\varsi{i}\sbv{i}}f(\varx)=-\langle\nabla f(\varx),\sbv{i}\rangle>0$ due to the linearity of the scalar product. Hence, $f$ is directional monotonically increasing in direction $\varsi{i}\sbv{i}$.\qedhere
    \end{description}
\end{proof}

Suppose we have $\failureprob:\reals^n\rightarrow\reals$ with $\failureprob\in\diffbar{\reals^n}$, then according to Lemma~\ref{lem:directionalExistance} there exists a set of directions $\directions_{\varx}=\{\varsi{1}\sbv{1},\dots,\varsi{n}\sbv{n}\}$, with $\vars\in\{-1,1\}^n$, for each $\varx\in\reals^n$ such that $\failureprob$ is directional monotonically increasing on $\{\varx\}$ w.r.t.\ $\directions_{\varx}$.
There are only finitely many of these direction sets $\directions_j$, for $j\in\setft{1}{J}$, and for each of them we define the corresponding domain $\domain_j=\left\{\varx\in\reals^n\mid \directions_{\varx}=\directions_j\right\}$. These $\domain_j$ form a partition of $\reals^n$.
The Examples~\ref{ex:directionalNormal} to \ref{ex:directionalGamma} give an impression of what the resulting sets $\domain_j$ can look like, showing that their structure is quite simple.
These $\domain_j$ and their corresponding directions $\directions_j$ are then employed to determine in which direction the currently considered parameters $\param$ have to be rounded to underestimate their corresponding failure probability $\failureprob(\param)$.
For this purpose, we define the following rounding function, which incorporates this directional information when rounding the parameters.


\begin{definition}
    \label{def:emittingCone}
    Let $\directions=\{r_1,\dots,r_m\}\subset\reals^n$ be finite and $\varx\in\reals^n$, then we define the \emph{cone in direction $\directions$ emitting from $\varx$} to be
    \begin{equation*}
        \emittingcone{\varx}{\directions}\coloneqq\left\{\varx+\sum\limits_{i=1}^{m}\lambda_i r_i\mid\lambda\in\pos{\reals}^m\right\}.
    \end{equation*}
\end{definition}

\begin{definition}
    \label{def:floorTo}
    Let $\discr\subset\reals^n$ be a finite set and $\failureprob$ a failure probability function.
    Moreover, let $\failureprob$ induce a collection of sets $\domain_j\subseteq\paramspace$, $\directions_j\subset\reals^n$, for $j\in\setft{1}{J}$, such that $\failureprob$ is directional monotonically increasing on $\domain_j$ w.r.t.\ $\directions_j$.
    Furthermore, let $\param\in\paramspace$ and let $\domain_c$ be the set containing $\param$ with corresponding $\directions_c$. Then we define rounding function $\floorto{\cdot}{\discr}$ as follows:\
    \begin{equation*}
        \floorto{\cdot}{\discr}:\paramspace\rightarrow\discr,\,\,\floorto{\param}{\discr}\coloneqq\argmin\limits_{\otherparam\in\discr\cap\emittingcone{\param}{-\directions_c}}\left\{\eucldist{\otherparam -\param}\right\}.
    \end{equation*}
\end{definition}
This function ensures that the parameters get rounded to the nearest element of the discretization that lies in a direction in which $\failureprob$ decreases. This will later guarantee that the failure probability is consistently underestimated.

\section{Approximating the Error}
\label{sec:approximating}
As mentioned in Section~\ref{sec:graph}, the \problemabbrv\ can be represented by the flow problem induced by the \completegraphname, while solving the flow problem induced by the \graphname, based on some discretization $\discr$, only provides an approximation to the problem.
Now, we estimate the error that arises between the vehicle rotations in the \completegraphname\ and their counterparts in the \graphname\ when $\floorto{\cdot}{\discr}$ is applied during the construction of the \graphname. We then show that the refinement of $\discr$ leads to a sequence of approximations whose solution values approach the value of an optimal solution for \problemabbrv\ from below. Thus, the solutions of the LP relaxations of the corresponding formulation (\hyperref[eq:objective]{AP}) provide a sequence of lower bounds for \problemabbrv\ for increasingly finer discretizations.

In this section, we assume that the degradation functions $\degradation{\trip}$ associated with the trips $\trip\in\trips$ and the other operations, i.e., deadheading or waiting, are Lipschitz continuous. Furthermore, we suppose that $\failureprob$ is also Lipschitz continuous and that there exist $\domain_j\subseteq\reals^n$ and $\directions_j=\{\varsi{j,1}\sbv{1},\dots,\varsi{j,n}\sbv{n}\}$, for $j\in\setft{1}{J}$ and some $\varsi{j}\in\{-1,1\}^n$, such that $\failureprob$ is directional monotonically increasing on $\domain_j$ w.r.t.\ $\directions_j$.

\subsection{Bounding the Failure Probability}
In the following, we demand that the parameter space $\paramspace$ is a cuboid bounded by some upper and lower values on each component, i.e., $\exists\,l,u\in\reals^n$ such that
\begin{equation*}
    \begin{pmatrix}
        l_1 \\
        \vdots \\
        l_n
    \end{pmatrix}
    \preceq
    \begin{pmatrix}
        \param_1 \\
        \vdots \\
        \param_n
    \end{pmatrix}
    \preceq
    \begin{pmatrix}
        u_1 \\
        \vdots \\
        u_n
    \end{pmatrix}
    \quad\forall\param\in\paramspace,
\end{equation*}
where $\preceq$ denotes componentwise less than or equal to. Furthermore, w.l.o.g.\ we assume $l_i=0$ and $u_i=1$ for all $i\in\setft{1}{n}$.
Next, let $\cube{n}$ be the $n$-cube with side lengths one, then we choose its vertices $V(\cube{n})=\binary^n$ as initial discretization $\discr_{0}$.
We refine this discretization by iteratively decomposing the current cubes into subcubes. For this purpose, we insert $k\in\strictpos{\integer}$ equidistant points on each diagonal of the cubes and consider the vertices of all $(k+1)^n$ resulting subcubes as discretization $\discr_1$.
Subsequently, the procedure is repeated and each of the resulting subcubes is divided again to obtain $\discr_2$.
In this way, a family of discretizations $\{\discr_{i}\}_{i\in\pos{\integer}}$ with $\discr_i=\{0,\frac{1}{k^i},\frac{2}{k^i},\dots,1\}^n$ is obtained.
Note that all $\discr_i$ contain the upper and the lower bound of the parameter space.

In addition, we have to take into account the subsets $\domain_j$ of $\paramspace$ on which $\failureprob$ is directional monotonically increasing w.r.t.\ different sets of directions. For this purpose, the points at which the boundaries of the $\domain_j$ intersect the ``grid lines'' that form the cubes of the current discretization must be added to $\discr_i$.
We denote these points by \emph{additional boundary points} and assume that the boundaries between the $\domain_j$ are hyperplanes whose normal vectors are parallel to one of the ``grid lines''.
This assumption is reasonable as the examples given in Section~\ref{sec:consistentlyRounding} exhibit exactly this behavior.
Thus, the boundary hyperplanes divide the cubes of the discretization into subcubes on which $\failureprob$ possesses a uniform directional monotonicity behavior.
A discretization $\discr$ that possesses these properties is called \emph{suitable} in the following.
Recall Examples~\ref{ex:directionalNormal} and \ref{ex:directionalWeibull}:\ The failure probability of the family of normal distributions is directional monotonically increasing on $\{(\mu,\variance)\in\reals^2\mid\mu\geq 0,\variance>0\}$ w.r.t.\ $\{-\sbv{1},\sbv{2}\}$ and on $\{(\mu,\variance)\in\reals^2\mid\mu\leq 0,\variance>0\}$ w.r.t.\ $\{-\sbv{1},-\sbv{2}\}$. Thus, the points at which the ``grid lines'' of $\discr_i$ intersect the line $\left\{(0,\variance)\mid\variance>0\right\}$ must be added to $\discr_i$.
The same applies to the line $\left\{(\kappa,\alpha)\mid \kappa>0, \alpha\text{ fixed}\right\}$ when considering Weibull distributions.

Now, we show that for a point $\param$ in a cube and a set of directions, there is always a vertex of the cube in the corresponding cone of directions to which $\param$ can be rounded.

\begin{lemma}
    \label{lem:nCubeCorners}
    Let $\cube{n}=[0,1]^n$ be the $n$-dimensional unit cube, $\param\in\cube{n}$ arbitrary, $s=\in\{-1,1\}^n$ a vector of signs, and $\directions=\{\varsi{1}\sbv{1},\dots,\varsi{n}\sbv{n}\}$.
    Then, $\emittingcone{\param}{\directions}$ contains at least one vertex of $\cube{n}$.
\end{lemma}
\begin{proof}
    The vertices of $\cube{n}$ are $\nodes(\cube{n})=\binary^n$.
    Now, define $\node$ with\linebreak $\node_i=\frac{\varsi{i}+1}{2}\in\binary$ for all $i\in\setft{1}{n}$, then we have $\node\in \nodes(\cube{n})$.
    To show that $\node$ is contained in $\emittingcone{\param}{R}$, we have to find conic coefficients $\lambda_i\geq 0$ such that $\node = \param + \sum_{i=1}^{n}\lambda_i \varsi{i} \sbv{i}$, i.e., $\node_i = \param_i + \lambda_i \varsi{i}$.
    Depending on $\varsi{i}$, we define $\lambda_i\coloneqq\abs{\node_i-\param_i}\geq 0$ for all $i\in\setft{1}{n}$. Since $\param_i\in[0,1]$ we obtain
    \begin{description}
        \item[Case 1:] If $\varsi{i}=1$:\ $\param_i+\varsi{i}\lambda_i=\param_i+1-\param_i=1=\node_i$.
        \item[Case 2:] If $\varsi{i}=-1$:\ $\param_i+\varsi{i}\lambda_i=\param_i-(\param_i-0)=0=\node_i$.
    \end{description}
    Hence, the defined $\lambda_i$ are the desired conic coefficients and we conclude $\node\in \emittingcone{\param}{\directions}$.
\end{proof}
Note that one or more of the additional boundary points could also lie in $\emittingcone{\param}{\directions}$. In this case, these elements of the discretization would be closer to $\param$ than $\node$, since $\node$ is a vertex of the cube itself and thus has the largest distance to $\param$ in the corresponding direction. Therefore, applying $\floorto{\param}{\discr_i}$ would result in one of these additional points.

Next, we prove an upper bound on the approximation error that arises when applying the rounding function $\floorto{\cdot}{\discr_i}$.

\begin{lemma}
    \label{lem:paramError}
    Consider $\{\discr_{i}\}_{i\in\pos{\integer}}$ as defined above and let $\domain_j\subseteq\paramspace$, $\directions_j\subset\reals^n$, for $j\in\setft{1}{J}$, be sets such that $\failureprob$ is directional monotonically increasing on $\domain_j$ w.r.t.\ $\directions_j$. Then, the error resulting from the application of $\floorto{\cdot}{\discr_{i}}$ is $\error_{i}\coloneqq\max\limits_{\param\in\paramspace}\{\eucldist{\param-\floorto{\param}{\discr_{i}}}\} \leq \frac{\sqrt{n}}{k^{i}}$.
\end{lemma}
\begin{proof}
    Each discretization $\discr_{i}$ induces a partition of $\paramspace = [0,1]^n$ into cubes with identical side lengths, i.e.,
    $\mathcal{P}=\bigcup_{l=1}^{k^{ni}}\mathcal{I}_{l}$, where each $\mathcal{I}_{l}$ is a cube with side length $\frac{1}{k^i}$, for $l=1,...,k^{ni}$.
    Now, let $\param\in\paramspace$ be any parameter value, then there exists some $\mathcal{I}_c$ that contains $\param$, since $\mathcal{P}$ is a partition of $\paramspace$.
    Furthermore, the $\domain_j$ also give a partition of $\paramspace$, and $\failureprob$ is directional monotonically increasing on $\domain_j$ w.r.t.\ $\directions_j$.
    Let $\domain_c$ be the set such that $\param\in\domain_c$, then due to Lemma~\ref{lem:nCubeCorners} at least one vertex of $\mathcal{I}_c$ is contained in $\emittingcone{\param}{-\directions_c}$. Hence, we have $\nodes(\mathcal{I}_c)\cap\emittingcone{\param}{-\directions_c}\neq\emptyset$ and it follows $\discr_i\cap\emittingcone{\param}{-\directions_c}\neq\emptyset$.
    Thus, there exists $\otherparam\coloneqq\floorto{\param}{\discr_i}=\argmin_{\otherparam\in\discr_i\cap\emittingcone{\param}{-\directions_c}}\left\{\eucldist{\otherparam -\param}\right\}$, which is either a vertex of $\mathcal{I}_c$ or one of the additional boundary points.
    Since all of these points are located on the facets of $\mathcal{I}_c$, we obtain $\eucldist{\otherparam -\param}\leq\mathit{dia}(\mathcal{I}_c)=\frac{\sqrt{n}}{k^i}$.
    Hence, it follows $\error_{i}\leq\frac{\sqrt{n}}{k^i}$.
\end{proof}
Lemma~\ref{lem:paramError} shows that the error caused by the application of $\floorto{\cdot}{\discr}$ depends on the diameter of the greatest cube contained in $\discr$ and since $n$ and $k$ are fixed, we have $0\leq\lim_{i\rightarrow\infty} \error_{i} \leq \lim_{i\rightarrow\infty} \frac{\sqrt{n}}{k^i}\rightarrow 0$ for the described family of discretizations $\{\discr_{i}\}_{i\in\pos{\integer}}$.

Next, we estimate the rounding errors between the actual parameters of the health states and the occurring parameters in the \graphname\ caused by the utilization of the rounding function. For this purpose, we compare the parameter values that arise in the \completegraphname\ and the \graphname\ for corresponding vehicle rotations. Recall that the \completegraphname\ is based on a discretization that includes all potentially appearing parameters and thus exactly models both the \problemabbrv\ and the degradation.
Suppose we have a path $\orig$ in the \completegraphname\ that corresponds to a vehicle rotation, then we denote a path consisting of the same operations and trips in the \graphname\ as \emph{discrete counterpart} of $\orig$.
We show that such a counterpart exists and then estimate the occurring rounding errors along this path.

\begin{theorem}
    \label{thm:pathExistence}
    Consider a suitable discretization $\discr$ of $\paramspace$ w.r.t.\ $\failureprob$ and let $\orig$ be any path in the \completegraphname. Then, there exists a discrete counterpart $\rounded$ in the \graphname\ constructed w.r.t.\ $\discr$.
\end{theorem}
\begin{proof}
    It is sufficient to show that for any $\param\in\paramspace$ there is a parameter in $\discr$ to which $\param$ can be rounded by $\floorto{\cdot}{\discr}$.
    If $\paramspace$ is a cuboid, we can assume w.l.o.g.\ $\paramspace=[0,1]^n$ by scaling.
    A discretization that contains the bounds of $\paramspace$ and the additional boundary points induced by $\failureprob$ therefore consists at least of the elements contained in $\discr_0$, as described above.
    Thus, we can apply the idea of Lemma~\ref{lem:nCubeCorners} to $\discr$ and obtain an element to round to. Therefore, we can construct a discrete analogous $\rounded$ for each $\orig$.
\end{proof}

\begin{example}
    Suppose we have $\paramspace=[0,1]$ and a discretization $\discr=\binary$ consisting only of the upper and lower bound of $\paramspace$. Furthermore, assume that $\failureprob$ is directional monotonically increasing on $\paramspace$ w.r.t.\ $\directions=\left\{\sbv{1}\right\}$ and that there is no trip that would lead to a breakdown for a vehicle that is as good as new, i.e., $\degradation{\trip}(0)<1$ for all $\trip\in\trips$.
    Then, we have $\floorto{\param}{\discr}=0$ for each $\param\in\paramspace$, which leads to an extreme case of Theorem~\ref{thm:pathExistence}, since any path in the \completegraphname\ would be projected onto a discretized counterpart $\rounded$ where all parameter values are zero.
\end{example}

\begin{corollary}
    \label{cor:tripDegradation}
    Consider $\{\discr_{i}\}_{i\in\pos{\integer}}$ and an arbitrary trip $\trip\in\trips$. Let $L_{\trip}$ be a Lipschitz constant of $\degradation{\trip}$, $\param_{\orig}\in\paramspace$ a parameter value occurring in the \completegraphname, and $\param_{\rounded}=\floorto{\param_{\orig}}{\discr_{i}}\in\discr_{i}$ its discrete counterpart in the \graphname.
    Then, the error occurring after the operation of $\trip$ is bounded by $\eucldist{\degradation{\trip}(\param_{\orig})-\degradation{\trip}(\param_{\rounded})}\leq L_{\trip}\error_{i}$.
\end{corollary}
\begin{proof}
    Due to the Lipschitz continuity of $\degradation{\trip}$ and Lemma~\ref{lem:paramError} it follows
    \begin{equation*}
        \eucldist{\degradation{\trip}(\param_{\orig})-\degradation{\trip}(\param_{\rounded})}\leq L_{\trip}\eucldist{\param_{\orig}-\param_{\rounded}}=L_{\trip}\eucldist{\param_{\orig}-\floorto{\param_{\orig}}{\discr_{i}}}\leq L_{\trip}\error_{i}.\qedhere
    \end{equation*}
\end{proof}

After bounding the error that can occur due to the operation of a trip, we specify a bound on the error for multiple consecutive trips. Here, the error is propagated and amplified as $\floorto{\cdot}{\discr_{i}}$ is applied after each trip.

\begin{theorem}
    \label{thm:pathDegradation}
    Let $\orig$ be a path in the \completegraphname\ and $\rounded$ its discretized counterpart in the \graphname\ constructed w.r.t.\ $\discr_{i}$ for some $i\in\pos{\integer}$. Additionally,\linebreak let $q$ be the number of trips contained in these paths, and\linebreak $L=\max_{\trip\in\trips}\left\{L_{\trip}\mid L_{\trip}\text{ Lipschitz constant of }\degradation{\trip}\right\}$.
    Then, the error between the parameter value of $\orig$ and its discretized counterpart $\rounded$ after the $\timept$\textsuperscript{th} service is bounded by
    \begin{equation*}
        E_{k}\coloneqq\eucldist{\param_{\orig,\timept}-\param_{\rounded,\timept}}\leq\frac{(L^{\timept+1}-1)\,\error_{i}}{L-1},\quad\forall \timept\in\{0,\dots,q\}.
    \end{equation*}
\end{theorem}
\begin{proof}
    In the following, $\param_{\orig,\timept}\in\paramspace$ and $\param_{\rounded,\timept}\in\discr_i$ are the parameter values of $\orig$ and $\rounded$ after the $\timept$\textsuperscript{th} service, respectively.
    W.l.o.g.\ we assume that the services are indexed in the order of their operation.
    
    Suppose a certain vehicle $\vehicle\in\vehicles$ operates the schedule corresponding to path $\orig$. Then, the parameters of the initial health state of $\vehicle$ determine the initial parameters of $\orig$, i.e., $\param_{\orig,0}=\param_{\vehicle,0}$. Therefore, the initial parameters of $\rounded$ are given by $\param_{\rounded,0}=\floorto{\param_{\orig,0}}{\discr_{i}}$ and by Lemma~\ref{lem:paramError} we obtain\linebreak $E_{0}=\eucldist{\param_{\orig,0}-\param_{\rounded,0}}\leq\error_{i}$.
    In addition, for each $\timept\in\setft{1}{q}$ it holds:\
    \begin{align*}
        E_{\timept}&=\eucldist{\param_{\orig,\timept}-\param_{\rounded,\timept}} \\
        &=\eucldist{\degradation{\trip_{\timept}}(\param_{\orig,\timept-1})-\floorto{\degradation{\trip_{\timept}}(\param_{\rounded,\timept-1})}{\discr_i}} \\
        &=\eucldist{\degradation{\trip_{\timept}}(\param_{\orig,\timept-1})-\degradation{\trip_{\timept}}(\param_{\rounded,\timept-1})+\degradation{\trip_{\timept}}(\param_{\rounded,\timept-1})-\floorto{\degradation{\trip_{\timept}}(\param_{\rounded,\timept-1})}{\discr_i}} \\
        &\leq\eucldist{\degradation{\trip_{\timept}}(\param_{\orig,\timept-1})-\degradation{\trip_{\timept}}(\param_{\rounded,\timept-1})}+\eucldist{\degradation{\trip_{\timept}}(\param_{\rounded,\timept-1})-\floorto{\degradation{\trip_{\timept}}(\param_{\rounded,\timept-1})}{\discr_i}} \\
        &\leq L_{\timept}\eucldist{\param_{\orig,\timept-1}-\param_{\rounded,\timept-1}}+\error_{i} \\
        &=L_{\timept}\cdot E_{\timept-1}+\error_{i}.
    \end{align*}
    In the penultimate transformation, we applied the Lipschitz continuity of $\degradation{\trip_{\timept}}$ and Lemma~\ref{lem:paramError} for $\degradation{\trip_{\timept}}(\param_{\rounded,\timept-1})\in\paramspace$.
    Next, we show by induction on $\timept$ that the stated formula for $E_{\timept}$ holds:\
    \begin{description}
        \item[BC:] When $\timept=0$, we have by Lemma~\ref{lem:paramError}:\
        \begin{equation*}
            E_{0}=\eucldist{\param_{\orig,0}-\param_{\rounded,0}}\leq\error_{i}=\frac{(L^1-1)\,\error_{i}}{L-1}
        \end{equation*}
        \item[IH:] Assume $E_{l}\leq\frac{(L^{l+1}-1)\,\error_{i}}{L-1}$ holds for all $l\in\{0,\dots,\timept\}$.
        \item[IS:] Consider $\timept+1$, then using $L_{\timept+1}\leq L$, it follows:\
        \begin{align*}
            E_{\timept+1}&=L_{\timept+1}\cdot E_{\timept}+\error_{i} \\
            &\leq L\cdot E_{\timept}+\error_{i} \\
            &\leq L\cdot \frac{(L^{\timept}-1)\,\error_{i}}{L-1}+\error_{i} \\
            &=\left(\frac{L^{\timept+1}-L}{L-1}+1\right)\cdot\error_{i} \\
            &=\left(\frac{L^{\timept+1}-L}{L-1}+\frac{L-1}{L-1}\right)\cdot\error_{i} \\
            &=\frac{(L^{\timept+1}-1)\,\error_{i}}{L-1}
        \end{align*}
    \end{description}
    Thus, the claim follows.
\end{proof}

Note that, although we do not consider maintenance operations in Theorem~\ref{thm:pathDegradation}, they can only reduce the resulting error since they reset the parameters to some predefined value $\param_{\maintenance}\in\paramspace$. Therefore, by Lemma~\ref{lem:paramError}, the resulting error is at most $\error_{i}$ and thus no further error is propagated from the services that were operated before the maintenance.

Furthermore, since every timetable consists of only a finite number of services, i.e., $\abs{\trips}\in\strictpos{\integer}$, there exists a maximum number of services $q\leq\abs{\trips}$ that can be performed consecutively by any vehicle. Therefore, the maximum error that would be possible for each path is $E_{\max}(i)\leq\frac{(L^{q+1}-1)\,\error_{i}}{L-1}$, where $L$ and $q$ are constants. Thus, $\lim\limits_{i\rightarrow\infty}E_{\max}(i)\rightarrow0$.
Although $q$ and $L$ can be rather large in real world instances, the error depends on the granularity of the considered discretization, and therefore the refinement procedure determines the rate of convergence, which can be exponentially fast.
An example of the error propagation for a one-dimensional parameter space is depicted in Figure~\ref{fig:errorPropagation}.

\begin{figure}
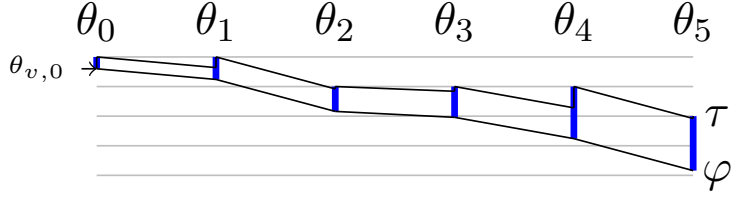

    \centering
    \include{errorApproxPicture.tex}
    \caption{Example to illustrate the error propagation. The lower path $\orig$ is the one using continuous parameter values, while the upper path $\rounded$ is based on the discretization given by the gray lines. The errors $E_{\timept}$ are highlighted in blue, and $\param_{\vehicle,0}$ is the initial parameter value of the utilized vehicle $\vehicle$.}
    \label{fig:errorPropagation}
\end{figure}

\begin{corollary}
    \label{cor:probError}
    Let $\orig$ be a path in the \completegraphname\ and $\rounded$ be its discretized counterpart in the \graphname\ constructed w.r.t.\ $\discr_{i}$ for some $i\in\pos{\integer}$.
    Furthermore,\linebreak let $q$ be the number of trips contained in these paths,\linebreak $L=\max_{\trip\in\trips}\left\{L_{\trip}\mid L_{\trip}\text{ Lipschitz constant of }\degradation{\trip}\right\}$, and $\failureprob$ a Lipschitz continuous function with Lipschitz constant $L_{\prob}$.
    Then, the error between the failure probability of $\orig$ and its discretized counterpart $\rounded$ after the $\timept$\textsuperscript{th} service is bounded by
    \begin{equation*}
        \eucldist{\failureprob(\param_{\orig,\timept})-\failureprob(\param_{\rounded,\timept})}\leq\frac{L_{\prob}(L^{\timept+1}-1)\,\error_{i}}{L-1},\quad\forall \timept\in\setft{0}{q}.
    \end{equation*}
\end{corollary}
\begin{proof}
    By applying the Lipschitz continuity of $\failureprob$ and Theorem~\ref{thm:pathDegradation}, we conclude
    \begin{equation*}
        \eucldist{\failureprob(\param_{\orig,\timept})-\failureprob(\param_{\rounded,\timept})}\leq L_{\prob}\eucldist{\param_{\orig,\timept}-\param_{\rounded,\timept}}\leq\frac{L_{\prob}(L^{\timept+1}-1)\,\error_{i}}{L-1}.\qedhere
    \end{equation*}
\end{proof}

Corollary~\ref{cor:probError} shows that the error of the failure probability of any path further approaches zero with each refinement of the discretization.
In addition to Corollary~\ref{cor:probError}, it holds $\eucldist{\failureprob(\param_{\orig,\timept})-\failureprob(\param_{\rounded,\timept})}\leq1$, since $\failureprob$ is based on a cumulative distribution function and can therefore only take values in $[0,1]$.
Hence, the cost error in the objective function of the approximate problem w.r.t.\ a given path $\orig$ is at most $c_{F}\cdot q$, where $q$ is the number of services contained in $\orig$.

\subsection{Approaching the Optimum from Below}
Next, we show that the utilization of $\floorto{\cdot}{\discr}$ leads to a consistent underestimation of the parameters in the \graphname.
Consequently, the failure probability is also underestimated and thus each path in the \graphname\ has lower costs than the corresponding path in the \completegraphname. Therefore, the objective value of a solution to the approximate problem is less than or equal to the problem induced by the \completegraphname, i.e., the \problemabbrv\ itself.
This observation allows us to derive a method that approaches the optimal value of a solution for \problemabbrv\ from below and can therefore be used to derive lower bounds or for a dual solution approach.

Recall that we have intentionally included the additional boundary\linebreak points in the discretizations $\{\discr_{i}\}_{i\in\pos{\integer}}$. Since these points are closer to the interior points of each cube than the vertices of the cube, the parameters are rounded to these values if they lie in the cone emanating from the parameter in question, before being rounded to other elements of $\discr$.
This prevents the rounding function from mapping parameters to sets on which $\failureprob$ has a different directional monotonicity.

From now on, we additionally require that the $\domain_j$ induced by the failure probability $\failureprob$ possess an order and that we have $\failureprob(\param)\leq\failureprob(\otherparam)$ for $\param\in\domain_j$, $\otherparam\in\domain_{j+1}$, and $j\in\setft{1}{J-1}$, where $J$ is the number of domains with different directional monotonicity.
As before, this assumption is reasonable since the families of normal and Weibull distributions exhibit exactly this behavior, see Examples~\ref{ex:monotonicNormal} and \ref{ex:monotonicWeibull}, and the family of gamma distributions has the same directional monotonicity on the entire parameter space $\paramspace$.

\begin{example}
    \label{ex:monotonicNormal}
    For the normal distribution, we have $\paramspace=\domain_1\cup\domain_2$ with\linebreak $\domain_1=\{(\mu,\variance)\in\reals^2\mid\mu\geq 0,\variance>0\}$, $\domain_2=\{(\mu,\variance)\in\reals^2\mid\mu\leq 0,\variance>0\}$ and failure probability \eqref{eq:failureProbNormal}.
    \begin{description}
        \item[$\domain_1$:] $\mu\geq 0$, $\variance >0$:\
        \begin{equation*}
            \failureprob(\mu,\variance)=\frac{1}{2}\Bigg(1+\overbrace{\erf\Bigg(\underbrace{\frac{-\mu}{\sqrt{2\variance}}}_{\leq\,0}\Bigg)}^{\in\,[-1,0]}\Bigg)\in\left[0,\tfrac{1}{2}\right]
        \end{equation*}
        \item[$\domain_2$:] $\mu\leq 0$, $\variance >0$:\
        \begin{equation*}
            \failureprob(\mu,\variance)=\frac{1}{2}\Bigg(1+\overbrace{\erf\Bigg(\underbrace{\frac{-\mu}{\sqrt{2\variance}}}_{\geq\,0}\Bigg)}^{\in\,[0,1]}\Bigg)\in\left[\tfrac{1}{2},1\right]
        \end{equation*}
    \end{description}
    Hence, we obtain $\failureprob(\param)\leq\failureprob(\otherparam)$ for $\param\in\domain_1$, $\otherparam\in\domain_2$.
\end{example}

\begin{example}
    \label{ex:monotonicWeibull}
    For the Weibull distribution, we have $\paramspace=\domain_1\cup\domain_2$ with $\domain_1=\{(\kappa,\lambda)\in\strictpos{\reals}^2\mid\lambda\geq\alpha\}$, $\domain_2=\{(\kappa,\lambda)\in\strictpos{\reals}^2\mid\lambda\leq\alpha\}$ and failure probability \eqref{eq:failureProbWeibull}.
    \begin{description}
        \item[$\domain_1$:] $\kappa >0$, $\lambda\geq\alpha$:\
        \begin{equation*}
            \failureprob(\alpha,\kappa,\lambda)=1-\exp\Big(\overbrace{-\Big(\underbrace{\frac{\alpha}{\lambda}}_{\leq\,1}\Big)^{\kappa}}^{\in\,[-1,0]}\Big)\in\left[0,\tfrac{e-1}{e}\right]
        \end{equation*}
        \item[$\domain_2$:] $\kappa >0$, $\lambda\leq\alpha$:\
        \begin{equation*}
            \failureprob(\alpha,\kappa,\lambda)=1-\exp\Big(\overbrace{-\Big(\underbrace{\frac{\alpha}{\lambda}}_{\geq\,1}\Big)^{\kappa}}^{\leq\,-1}\Big)\in\left[\tfrac{e-1}{e},1\right]
        \end{equation*}
    \end{description}
    Hence, we obtain $\failureprob(\param)\leq\failureprob(\otherparam)$ for $\param\in\domain_1$, $\otherparam\in\domain_2$.
\end{example}

Furthermore, we assume in the following that all degradation functions are aligned with the considered failure probability, as specified in Definition~\ref{def:aligned}. This reflects the idea that new or recently maintained vehicles do not degrade faster than vehicles that are already worn out.

\begin{definition}
    \label{def:aligned}
    Let $\failureprob$ be the considered failure probability and $\degradation{\trip}$ the degradation function of some trip $\trip\in\trips$. Then we say that $\degradation{\trip}$ is \emph{aligned} with $\failureprob$ if it holds $\failureprob(\degradation{\trip}(\param))\leq\failureprob(\degradation{\trip}(\otherparam))$ for all $\param,\otherparam\in\paramspace$ with $\failureprob(\param)\leq\failureprob(\otherparam)$.
\end{definition}

\begin{definition}
    \label{def:signedCompwiseLessEqual}
    Let $a,b\in\reals^n$ and $\vars\in\{-1,1\}^n$, then we write $a\sgnleq{\vars} b$ to indicate that $a$ is \emph{componentwise less than or equal to $b$ w.r.t.\ signs $\vars$} if $a_i\leq b_i$ if $\varsi{i}=1$ and $a_i\geq b_i$ if $\varsi{i}=-1$ holds for all $i\in\setft{1}{n}$.
\end{definition}

\begin{lemma}
    \label{lem:applyRounding}
    Consider a suitable discretization $\discr$ of $\paramspace$ w.r.t.\ $\failureprob$ and let $\param\in\domain\subseteq\paramspace$, where $\failureprob$ is directional monotonically increasing on $\domain$ w.r.t.\ $\directions=\{\varsi{1}\sbv{1},\dots,\varsi{n}\sbv{n}\}$ with $\vars\in\{-1,1\}^n$. Then, it holds $\floorto{\param}{\discr}\sgnleq{\vars}\param$.
\end{lemma}
\begin{proof}
    Analogous to the proof of Theorem~\ref{thm:pathExistence}, we assume w.l.o.g.\ that\linebreak $\paramspace=[0,1]^n$ and that $\discr$ contains at least the elements of $\discr_0$.
    By Lemma~\ref{lem:nCubeCorners} we then have $\discr\cap\emittingcone{\param}{-\directions}\neq\emptyset$ and set $\otherparam\coloneqq\floorto{\param}{\discr}=\argmin_{\psi\in\discr\cap\emittingcone{\param}{-\directions}}\{\eucldist{\psi-\param}\}$.
    Then, the definition of $\emittingcone{\param}{-\directions}$, i.e., Definition~\ref{def:emittingCone}, implies that there exists $\lambda\in\pos{\reals}^n$ such that $\otherparam=\param-\sum_{i=1}^{n}\lambda_i \varsi{i} \sbv{i}$.
    It follows that $\otherparam_i\leq\param_i$ if $\varsi{i}=1$ and $\otherparam_i\geq\param_i$ if $\varsi{i}=-1$. From this we conclude $\floorto{\param}{\discr}=\otherparam\sgnleq{\vars}\param$.
\end{proof}

\begin{corollary}
    \label{cor:applyRoundingProb}
    Consider a suitable discretization $\discr$ of $\paramspace$ w.r.t.\ $\failureprob$ and let $\param\in\domain\subseteq\paramspace$, where $\failureprob$ is directional monotonically increasing on $\domain$ w.r.t.\ $\directions=\{\varsi{1}\sbv{1},\dots,\varsi{n}\sbv{n}\}$ with $\vars\in\{-1,1\}^n$. Then, it holds $\failureprob(\floorto{\param}{\discr})\leq\failureprob(\param)$.
\end{corollary}
\begin{proof}
    Apply Lemma~\ref{lem:applyRounding} and the directional monotonicity of $\failureprob$ on $\domain$ w.r.t.\ $\directions$.
\end{proof}

\begin{theorem}
    \label{thm:underestimateProb}
    Consider a suitable discretization $\discr$ of $\paramspace$ w.r.t.\ $\failureprob$ and let $\orig$ be a path in the \completegraphname, and $\rounded$ its discretized counterpart in the \graphname.
    Then, it holds $\failureprob(\param_{\rounded,\timept})\leq\failureprob(\param_{\orig,\timept})$ for all $k\in\{0,\dots,q\}$, where $q$ is the number of services contained in $\orig$.
\end{theorem}
\newpage
\begin{proof}
    Apply induction on $\timept$:\
    \begin{description}
        \item[BC:] When $\timept=0$, we have by Corollary~\ref{cor:applyRoundingProb}:\
        \begin{equation*}
            \failureprob(\param_{\rounded,0})=\failureprob(\floorto{\param_{\orig,0}}{\discr})\leq\failureprob(\param_{\orig,0}).
        \end{equation*}
        \item[IH:] Assume $\failureprob(\param_{\rounded,l})\leq\failureprob(\param_{\orig,l})$ holds for all $l\in\{0,\dots,\timept\}$.
        \item[IS:] Consider $\timept+1$, then it follows:\

        By Corollary~\ref{cor:applyRoundingProb} we have
        \begin{equation*}
            \failureprob(\param_{\rounded,\timept+1})=\failureprob(\floorto{\degradation{\timept+1}(\param_{\rounded,\timept})}{\discr})\leq\failureprob(\degradation{\timept+1}(\param_{\rounded,\timept})).
        \end{equation*}
        In addition, using the induction hypothesis $\failureprob(\param_{\rounded,\timept})\leq\failureprob(\param_{\orig,\timept})$ and the assumption that $\degradation{\timept+1}$ is aligned with $\failureprob$, we obtain
        \begin{equation*}
            \failureprob(\degradation{\timept+1}(\param_{\rounded,\timept}))\leq\failureprob(\degradation{\timept+1}(\param_{\orig,\timept})).
        \end{equation*}
        Thus, we conclude
        \begin{align*}
            \failureprob(\param_{\rounded,\timept+1})
            &=\failureprob(\floorto{\degradation{\timept+1}(\param_{\rounded,\timept})}{\discr})\\
            &\leq\failureprob(\degradation{\timept+1}(\param_{\rounded,\timept}))\\
            &\leq\failureprob(\degradation{\timept+1}(\param_{\orig,\timept}))\\
            &=\failureprob(\param_{\orig,\timept+1})\qedhere
        \end{align*}
    \end{description}
\end{proof}

\subsection{An Iterative Refinement Approach}
Lemma~\ref{lem:applyRounding} and Theorem~\ref{thm:underestimateProb} guarantee that the parameter values in the \graphname\ are always underestimated when $\floorto{\cdot}{\discr}$ is employed during its construction. This leads to an approximation where the wear applied to the vehicles is less severe than in the \completegraphname\ and thus in the original \problemabbrv\ instance. Due to the underestimated parameters, the failure probability is underestimated as well and the resulting solution for (\hyperref[eq:objective]{AP}) has lower costs than the corresponding solution for the original instance.
By Theorem~\ref{thm:pathExistence}, there is a discrete counterpart in the \graphname\ for each path in the \completegraphname\ if the \graphname\ is constructed based a suitable discretization. Since these paths represent all feasible vehicle rotations, the value of an optimal solution of the approximate problem is a lower bound to the value of an optimal solution of the original instance.

Next, we show that adding values to a suitable discretization can only increase the objective value of a solution for the approximate problem. Thus, the subsequent refinement of the discretization leads to a sequence of solutions whose approximate objective values continue to increase and approach the optimal value of a solution for \problemabbrv\ from below.
For this purpose, we use the notion of \emph{corresponding paths}, similar to the discrete counterpart of a path, as above:\ Suppose we are given two \graphname s $\graph_1$ and $\graph_2$ that are constructed based on suitable discretizations $\discr_1$ and $\discr_2$, respectively, and that approximate the same problem. Furthermore, let $\rounded_1$ be a path in $\graph_1$, then we denote a path in $\graph_2$ that consists of the same sequence of services as the corresponding path to $\rounded_1$. Again, such a path exists by Theorem~\ref{thm:pathExistence}.

\begin{theorem}
    \label{thm:increasingPathCosts}
    Consider a suitable discretization $\discr_1$ of $\paramspace$ w.r.t.\ $\failureprob$ and set $\discr_2\coloneqq\discr_1\cup\{\param\}$ for some $\param\in\paramspace$. Let $\graph_1$ and $\graph_2$ be the \graphname s based on $\discr_1$ and $\discr_2$, respectively, and let $\rounded_1$ be a path in $\graph_1$ with $\rounded_2$ being the corresponding path in $\graph_2$.
    Then, it holds $\cost(\rounded_1)\leq\cost(\rounded_2)$.
\end{theorem}
\begin{proof}
    If $\param\in\discr_1$, we have $\discr_1=\discr_2$ and no path would get altered. Consequently, we have $\cost(\rounded_1)=\cost(\rounded_2)$.
    We therefore assume $\param\notin\discr_1$ in the following.
    Recall that we assume $\paramspace$ to be a bounded cuboid and that each discretization induces a partition of $\paramspace$ into cuboids, one of which contains $\param$. Let this cell be $\cube{c}$.
    
    Now, let $\arcs_c$ be the set of all arcs in $\graph_1$ that have a head node whose parameters are equal to a vertex of $\cube{c}$. Then, we obtain two cases:\
    First, if $\rounded_1$ does not contain an arc of $\arcs_c$, all of its arcs are also present in the \graphname\ w.r.t.\ $\discr_2$ and are not modified. It follows that the parameters of the nodes in $\rounded_2$ are equal to those in $\rounded_1$, resulting in $\cost(\rounded_1)=\cost(\rounded_2)$.
    Second, if $\rounded_1$ contains an arc of $\arcs_c$, let $\arc_c$ be the first of them. Then, each arc preceding $\arc_c$ remains unchanged. Let $\node_c$ be the tail node of $\arc_c$ and $\trip\in\trips$ the trip corresponding to $\arc_c$, then we determine $\param_r\coloneqq\degradation{\trip}(\param_c)$, where $\param_c$ is the parameter value of $\node_c$.
    Now, there are again two possibilities:\ If $\floorto{\param_r}{\discr_2}=\floorto{\param_r}{\discr_1}$, the failure costs associated with $\trip$ in $\graph_2$ remain the same and we consider the next arc of $\rounded_1$. But if $\floorto{\param_r}{\discr_1}\sgnleq{\vars}\floorto{\param_r}{\discr_2}=\param$, where $\vars$ is the vector of signs that determines in which directions $\failureprob$ is directional monotonically increasing on the domain containing $\param_r$, we get $\failureprob(\floorto{\param_r}{\discr_1})\leq\failureprob(\floorto{\param_r}{\discr_2})$. The same arguments apply to the subsequent arcs of $\arc_c$ in $\rounded_2$. Therefore, the parameter values of the nodes in $\rounded_2$ may be updated, but they can only increase compared to the values of the corresponding nodes in $\rounded_1$. Consequently, the failure probability can only increase, as can the costs. This results in $\cost(\rounded_1)\leq \cost(\rounded_2)$.
\end{proof}

\begin{corollary}
    \label{cor:increasingSolCosts}
    Consider a suitable discretization $\discr_1$ of $\paramspace$ w.r.t.\ $\failureprob$ and set $\discr_2\coloneqq\discr_1\cup\{\param\}$ for some $\param\in\paramspace$. Let $\graph_1$ and $\graph_2$ be the \graphname s based on $\discr_1$ and $\discr_2$, respectively.
    Then, it holds $\cost(\varx_1^*)\leq\cost(\varx_2^*)$, where $\varx_i^*$ is an optimal solution for the approximate problem induced by $\graph_i$ for $i\in\{1,2\}$.
\end{corollary}
\begin{proof}
    Suppose we have an optimal solution $\varx_2^*$ for the approximate problem induced by $\graph_2$. $\varx_2^*$ consists of a set of paths $\{\rounded_{2,1},\dots,\rounded_{2,m}\}$ that also exist in $\graph_1$ according to Theorem~\ref{thm:pathExistence}. Let these paths be $\rounded_{1,1},\dots,\rounded_{1,m}$.
    By Theorem~\ref{thm:increasingPathCosts}, we have $\cost(\rounded_{1,i})\leq\cost(\rounded_{2,i})$ for all $i\in\setft{1}{m}$. These paths yield a solution $\varx_1$ for the approximate problem induced by $\graph_1$ with $\cost(\varx_1)\leq\cost(\varx_2^*)$. Since we have $\cost(\varx_1^*)\leq\cost(\varx_1)$, the claim follows.
\end{proof}

Theorem~\ref{thm:increasingPathCosts} and Corollary~\ref{cor:increasingSolCosts} give rise to an iterative solution approach:\
Start with a suitable discretization $\discr$ of $\paramspace$ w.r.t.\ $\failureprob$, then construct the corresponding \graphname, solve the induced approximate problem (\hyperref[eq:objective]{AP}) and iteratively refine $\discr$. This yields a sequence of solutions whose objective values approach the value of an optimal solution for \problemabbrv\ from below.
Such a sequence of discretizations is given by the family $\{\discr_i\}_{i\in\pos{\integer}}$ described at the beginning of this section.
Furthermore, the approximate solutions can be transformed to solutions for the original problem. For this purpose, we consider the service sequences given by the paths contained in the solution and propagate the parameters through the exact degradation functions.
The resulting procedure is described in Algorithm~\ref{algo:iterativeRefinement} and leads to a dual solution approach for \problemabbrv.
However, if one is only interested in a lower bound for the instance, the algorithm can be modified by solving the LP relaxation of (AP) in line 6 and removing lines 1, 7, and 9.

\RestyleAlgo{ruled}
\DontPrintSemicolon
\begin{algorithm}
    \caption{Iterative Refinement Approach for \problemabbrv}
    \label{algo:iterativeRefinement}
    \KwData{\problemabbrv\ instance $\mathcal{I}$}
    \KwResult{Solution to \problemabbrv}
    $ub\gets\infty$\;
    $lb\gets -\infty$\;
    \While{$\text{time limit not reached}$ \upshape\textbf{and} $ lb < ub$}{
        $\discr\gets\discr_i$ as described in Section~\ref{sec:approximating}\;
        $\graph\gets$ \graphname\ of $\mathcal{I}$ w.r.t.\ $\floorto{\cdot}{\discr}$\;
        $c\gets$ solution of (\hyperref[eq:objective]{AP}) for $\graph$\;
        $x\gets$ solution corresponding to $c$ w.r.t.\ original degradation functions\;
        $lb\gets v(c)$\;
        $ub\gets v(x)$\;
    }
    return $x$\;
\end{algorithm}

\section{Computational Results}
\label{sec:results}
In this section, the results of Algorithm~\ref{algo:iterativeRefinement} for the test instances provided by \cite{prause2023construction} are presented and discussed.
The scenarios are based on real timetables of a private German railroad company and have a time horizon of one week. The health states of the vehicles represent the conditions of their doors and are assumed to be normally distributed, i.e., failure probability \eqref{eq:failureProbNormal} has to be utilized. In addition, the occurring degradation behavior is non-linear.

\subsection{Computational Setup}
The computations were conducted on a machine with Intel(R) Xeon(R) Gold 6342 @ 2.80GHz CPUs, eight cores, and 64GB of RAM.
The algorithms were implemented in Julia v1.9.4~\cite{julia} and Gurobi v10.0.2~\cite{gurobi} was used to solve the ILP and LP formulations.

\subsection{Results}
The results obtained for each of the instances are specified in Table~\ref{tab:results}.
The first four columns contain the characteristics of the instances, while column five to eight hold the obtained solution values and lower bounds. Here, MSH is the multi-swap heuristic presented in \cite{prause2023multi}, DA is the dual solution approach given in Algorithm~\ref{algo:iterativeRefinement}, and LP-LB is the lower bound derived from this algorithm by utilizing the LP relaxation of (AP). The best solution and the best lower bound are each marked in bold. The last column then contains the gap in percent between the best obtained solution and the best lower bound.

First, we describe the results obtained by MSH and DA regarding solutions to the test instances.
Here, the heuristic approach was able to achieve the best result for five of the six instances, and its solution for instance T4, where the dual approach obtained the best result, has a gap of only 0.01\% and is therefore almost equally good.
DA on the other hand, could only find the best solution for one instance, i.e., T4, and was not able to generate a primal feasible result for instance T6 at all. Furthermore, the generated solutions have a gap of 4.2 -- 70.5\% to the heuristic results for the instances where the heuristic achieved the best results.
Thus, MSH is a better choice when it comes to generating feasible solutions, since the heuristic consistently delivers better or almost equally good results as the dual approach.
This behavior was expectable as MSH makes the greatest progress within the first 400 seconds~\cite{prause2023multi}, in which DA is still at a point where the discretization is rather coarse leading to an approximation with less severe degradation. Hence, the solutions obtained by the dual approach tend to ignore maintenance at that point in time. In addition, DA relies on repeatedly solving ILP formulation (AP), whose size increases with each iteration.

Next, the results regarding the obtained lower bounds are discussed.
Here, LP-LB and DA achieved the same result for three of the six instances, i.e., T2, T5 and T6. LP-LB performs best on two of the instances (T1 and T3), where the results obtained by DA deviate from the LP-based lower bounds by 0.02\% and 1.5\%, respectively. In one case (T4), DA achieved the best lower bound, while the result of LP-LB is 0.5\% away from this bound.
These similarly good results are obtained because the LP relaxation is used as a lower bound in the solution process of the ILP within DA.
This leads us to two conclusions:\
On the one hand, this shows that the LP relaxation of (AP) is very tight, as the lower bound in the ILP formulations does not seem to benefit from the integrality constraints of its variables. On the other hand, LP-LB does not appear to benefit from the fact that it does not have to solve the ILP and thus can perform more iterations, allowing the consideration of finer discretizations at an earlier stage.
It therefore makes no difference whether DA or LP-LB is used for determining lower bounds.

In summary, a combination of MSH and LP-LB appears to achieve the best results. With this combination it is possible to generate solutions and lower bounds with gaps of less than 5\% for test instances originating from genuine timetables.
Although the dual solution approach was not able to obtain similarly good solutions as MSH, it performed well in generating lower bounds for \problemabbrv\ instances. Moreover, it was able to solve one instance, namely T5, to optimality.

\begin{table}
    \centering
    \caption{Characteristics and results for the test instances.}
    \label{tab:results}
    \resizebox{\linewidth}{!}{
    \begin{tabular}{ccccrrrrc}
        \toprule
        \multirow{2}{*}{Instance} & \multirow{2}{*}{Trips} & \multirow{2}{*}{Destinations} & \multirow{2}{*}{Vehicles} & \multicolumn{2}{c}{Solution Value} & \multicolumn{2}{c}{\quad Lower Bound} & \multirow{2}{*}{Best Gap in \%} \\
        &&&& \multicolumn{1}{c}{MSH} & \multicolumn{1}{c}{DA} & \multicolumn{1}{c}{DA} & \multicolumn{1}{c}{LP-LB} \\
        \midrule
        T1  & 566   & 8     & 6     & \best{269,728}    & 480,059           & 261,384           & \best{261,432}    & 3.08 \\
        T2  & 608   & 10    & 7     & \best{433,328}    & 452,201           & \best{428,349}    & \best{428,349}    & 1.15 \\
        T3  & 636   & 15    & 16    & \best{1,419,687}  & 4,817,547         & 1,360,954         & \best{1,381,725}  & 2.67 \\
        T4  & 679   & 9     & 8     & \best{196,411}    & 277,687           & \best{190,537}    & 189,577           & 2.99 \\
        T5  & 813   & 16    & 14    & 327,805           & \best{327,770}    & \best{327,770}    & \best{327,770}    & 0.00 \\
        T6  & 919   & 17    & 29    & \best{2,337,466}  & -                 & \best{2,290,596}  & \best{2,290,596}  & 2.01 \\
        \bottomrule
    \end{tabular}}
\end{table}

\section{Conclusion}
\label{sec:conclusion}
In this article, we have introduced the notions necessary to consistently underestimate the parameters during the construction of the \graphname\ for approximating the \problemabbrv\ when the considered health states are distributed by families of PDFs with more than one parameter.
For this purpose, we have constructed a rounding function that takes the domains into account on which the applied failure probability function is monotonically increasing w.r.t.\ different sets of directions.
Furthermore, we described a family of discretizations with increasing granularity that results in \graphname s with rising approximation quality, and presented different approaches to model the degradation based on the family of normal, Weibull, and gamma distributions.
We then estimated the error between vehicle rotations, i.e., paths, in the \graphname\ and \completegraphname, and it was proved that the objective value of a solution for the approximate problem can only underestimate the value of an optimal solution for \problemabbrv.
Hence, we derived a dual solution approach and a method to determine lower bounds for the discussed problem.
The final computations show the effectiveness of the lower bound, as the gap to heuristic results for real-world instances are rather small.

Possible next steps for future research are the implementation of a column generation approach for solving the LP relaxation of (AP) while computing the lower bound. This could lead to an acceleration of the solution process and might offer the possibility to solve the approximate problem for even finer discretizations than with the direct LP formulation.
In addition, this method could be employed to develop a branch and price approach for \problemabbrv. This could improve the efficiency of the dual approach and thus deliver better results.

\section*{Acknowledgements}
This work was supported by the innovation funding program ProFIT (grant no. 10174564) funded by the State of Berlin and co-funded by the European Union.

\bibliographystyle{plain}
\bibliography{references}

\end{document}

%% file: reductionGraph.tex
\resizebox{0.5\linewidth}{!}{
\begin{tikzpicture}
\tikzset{circ/.style={black,draw,circle,fill=zibblue,inner sep=0.5mm}}
		\node[circ] (0) at (0, 0) {$s$};
		\node[circ] (1) at (-2.5, 1) {$v_1^i$};
		\node[circ] (2) at (-1.5, 1) {$v_2^i$};
		\node[circ] (3) at (-0.5, 1) {$v_3^i$};
		\node[circ] (4) at (0.5, 1) {$v_4^i$};
		\node[circ] (5) at (1.5, 1) {$v_5^i$};
		\node[circ] (6) at (2.5, 1) {$v_6^i$};
		\node[circ] (7) at (-2.5, 5) {$v_1^f$};
		\node[circ] (8) at (-1.5, 5) {$v_2^f$};
		\node[circ] (9) at (-0.5, 5) {$v_3^f$};
		\node[circ] (10) at (0.5, 5) {$v_4^f$};
		\node[circ] (11) at (1.5, 5) {$v_5^f$};
		\node[circ] (12) at (2.5, 5) {$v_6^f$};
		\node[circ] (14) at (0, 6) {$t$};
		\node[circ] (15) at (-2.5, 2.33) {};
		\node[circ] (16) at (-2.5, 3.67) {};
		\node[circ] (17) at (-1.5, 3) {};
		\node[circ] (18) at (-0.5, 2.33) {};
		\node[circ] (19) at (-0.5, 3.67) {};
		\node[circ] (20) at (0.5, 2.33) {};
		\node[circ] (21) at (0.5, 3.67) {};
		\node[circ] (22) at (1.5, 2) {};
		\node[circ] (23) at (1.5, 3) {};
		\node[circ] (24) at (1.5, 4) {};
		\node[circ] (25) at (2.5, 3) {};
        \node[black] (a1) at (-2.3, 1.75) {$a$};
        \node[black] (d1) at (-2.3, 3) {$d$};
        \node[black] (g1) at (-2.3, 4.15) {$g$};
        \node[black] (a2) at (-1.3, 2.1) {$a$};
        \node[black] (d2) at (-1.3, 3.8) {$d$};
        \node[black] (d3) at (-0.3, 1.75) {$d$};
        \node[black] (e3) at (-0.3, 3) {$e$};
        \node[black] (g3) at (-0.3, 4.15) {$g$};
        \node[black] (c4) at (0.7, 1.75) {$c$};
        \node[black] (e4) at (0.7, 3) {$e$};
        \node[black] (f4) at (0.7, 4.15) {$f$};
        \node[black] (b5) at (1.7, 1.6) {$b$};
        \node[black] (c5) at (1.7, 2.45) {$c$};
        \node[black] (f5) at (1.7, 3.45) {$f$};
        \node[black] (g5) at (1.7, 4.3) {$g$};
        \node[black] (b6) at (2.7, 2.1) {$b$};
        \node[black] (g6) at (2.7, 3.8) {$g$};
		\draw[->,black] (0) to (1);
		\draw[->,black] (0) to (2);
		\draw[->,black] (0) to (3);
		\draw[->,black] (0) to (4);
		\draw[->,black] (0) to (5);
		\draw[->,black] (0) to (6);
		\draw[->,black] (1) to (15);
		\draw[->,black] (15) to (16);
		\draw[->,black] (16) to (7);
		\draw[->,black] (7) to (14);
		\draw[->,black] (2) to (17);
		\draw[->,black] (17) to (8);
		\draw[->,black] (8) to (14);
		\draw[->,black] (3) to (18);
		\draw[->,black] (18) to (19);
		\draw[->,black] (19) to (9);
		\draw[->,black] (9) to (14);
		\draw[->,black] (4) to (20);
		\draw[->,black] (20) to (21);
		\draw[->,black] (21) to (10);
		\draw[->,black] (10) to (14);
		\draw[->,black] (5) to (22);
		\draw[->,black] (22) to (23);
		\draw[->,black] (23) to (24);
		\draw[->,black] (24) to (11);
		\draw[->,black] (11) to (14);
		\draw[->,black] (6) to (25);
		\draw[->,black] (25) to (12);
		\draw[->,black] (12) to (14);
\end{tikzpicture}%
}

%% file: errorApproxPicture.tex
\resizebox{0.8\linewidth}{!}{
\begin{tikzpicture}
    \tikzset{
        dot/.style = {circle,black,fill, minimum size=0pt, inner sep=0pt, outer sep=0pt}
    }
    \tikzset{
        error/.style = {ultra thick,blue}
    }
    \node[black] (init) at (-0.5,0.9) {\tiny$\param_{\vehicle,0}$};
    \node[dot] (o0) at (0, 0.9) {};
    \node[dot] (o1) at (1, 0.81) {};
    \node[dot] (o2) at (2, 0.54) {};
    \node[dot] (o3) at (3, 0.49) {};
    \node[dot] (o4) at (4, 0.31) {};
    \node[dot] (o5) at (5, 0.04) {};
    \node[dot] (r0) at (0, 1) {};
    \node[dot] (rt1) at (1, 0.91) {};
    \node[dot] (r1) at (1, 1) {};
    \node[dot] (rt2) at (2, 0.73) {};
    \node[dot] (r2) at (2, 0.75) {};
    \node[dot] (rt3) at (3, 0.71) {};
    \node[dot] (r3) at (3, 0.75) {};
    \node[dot] (rt4) at (4, 0.57) {};
    \node[dot] (r4) at (4, 0.75) {};
    \node[dot] (rt5) at (5, 0.48) {};
    \node[dot] (r5) at (5, 0.5) {};

    \foreach \i in {0,...,4}{
        \draw[gray!50] (0,0.25*\i) -- (5,0.25*\i);
    }
    \foreach \i in {0,...,5}{
        \node[black] (l\i) at (\i,1.3) {$\param_{\i}$};
    }
    \node[black] (lo) at (5.2,0.0) {$\orig$};
    \node[black] (lr) at (5.2,0.5) {$\rounded$};
    \draw[error] (o0) -- (r0);
    \draw[error] (o1) -- (r1);
    \draw[error] (o2) -- (r2);
    \draw[error] (o3) -- (r3);
    \draw[error] (o4) -- (r4);
    \draw[error] (o5) -- (r5);
    \draw[->,black] (init) -- (o0);
    \draw[black] (o0) -- (o1);
    \draw[black] (o1) -- (o2);
    \draw[black] (o2) -- (o3);
    \draw[black] (o3) -- (o4);
    \draw[black] (o4) -- (o5);
    \draw[black] (o0) -- (r0);
    \draw[black] (r0) -- (rt1);
    \draw[black] (rt1) -- (r1);
    \draw[black] (r1) -- (rt2);
    \draw[black] (rt2) -- (r2);
    \draw[black] (r2) -- (rt3);
    \draw[black] (rt3) -- (r3);
    \draw[black] (r3) -- (rt4);
    \draw[black] (rt4) -- (r4);
    \draw[black] (r4) -- (rt5);
    \draw[black] (rt5) -- (r5);
\end{tikzpicture}%
}